\documentclass[a4paper,11pt]{article} 
\usepackage[italian,english]{babel}			
\usepackage{amsmath,amssymb,amsthm}
\usepackage{esint}
\usepackage{amsfonts} 
\usepackage{graphicx} 
\usepackage{braket}
\usepackage{times}				
\usepackage{ifthen}
\usepackage{xspace}
\usepackage{fancyhdr}
\usepackage{mathrsfs}
\usepackage{enumerate}
\usepackage{verbatim}
\usepackage{hyperref}
\usepackage[latin1]{inputenc}

\usepackage{color} 
\allowdisplaybreaks 
\definecolor{Red}{cmyk}{0,1,1,0.2}

\def\d{b_\Omega}

\def\L{{\rm Lip}}

\def \lo{\mathcal{P}_{m_0}^{\rm Lip}(\Gamma)}

\def \v{V}

\def \dg{\dot{\gamma}}
\def \g{\gamma}
\def \c{\lambda_p}
\def \cp{\lambda_+}
 \def\R{\mathbb R}
 \def \pl{p_\lambda}
 \def \plt{p_{2,\lambda}^{\tau}}
 \def \pln {p_{2,\lambda}^\nu}
 \def\pn{p^\nu}
 
 \def \pt{p^{\tau}}
 \def\pttx{D^\tau_xu(t,x)}
  
  \def\pnd{p^\nu_2}
  \def \qnd{q^\nu_2}
  \def \ptd{p^{\tau}_2}
  \def \qtd{q^{\tau}_2}
 \def \ho{H^{\tau}}

 \def\minimarc{\Gamma^*}
\usepackage{fancyhdr}

\usepackage{fancyhdr}
\newtheorem{definition}{Definition}[section] 
\theoremstyle{definition}

\theoremstyle{remark}
\newtheorem{remark}{Remark}[section]
\theoremstyle{plain}
\newtheorem{theorem}{Theorem}[section]
\newtheorem{lemma}{Lemma}[section]
\newtheorem{proposition}{Proposition}[section]
\newtheorem{corollario}{Corollary}[section]

\numberwithin{equation}{section}

\title{Mean Field Games with state constraints: from mild to pointwise solutions of the PDE system\footnote{This work was partly supported by the University of Rome Tor Vergata (Consolidate the Foundations 2015) and by the Istituto Nazionale di Alta Matematica ``F. Severi'' (GNAMPA 2016 Research Projects). The authors acknowledge the MIUR Excellence Department Project awarded to the Department of Mathematics, University of Rome Tor Vergata, CUP E83C18000100006. The second author is grateful to the Università Italo Francese (Vinci Project 2015). The work was also partly supported by the ANR (Agence Nationale de la Recherche) project ANR-16-CE40-0015-01.}}
\author{{\sc Piermarco Cannarsa}\thanks{%
		Dipartimento di Matematica, Universit\`{a} di Roma ``Tor Vergata" - {\tt cannarsa@mat.uniroma2.it}}\\
	{\sc Rossana Capuani}\thanks{%
		Dipartimento di Matematica, Universit\`{a} di Roma ``Tor Vergata" and CEREMADE, Universit\'e Paris~Dauphine - {\tt capuani@mat.uniroma2.it}}\\ {\sc Pierre Cardaliaguet}\thanks{%
		CEREMADE, Universit\'e Paris Dauphine - {\tt cardaliaguet@ceremade.dauphine.fr}}
	\\}
\date{}
\begin{document}
\maketitle
\begin{abstract}
Mean Field Games with state constraints are differential games with infinitely many agents, each agent facing a  constraint on his state. 
The aim of this paper is to provide a meaning of the PDE system associated with these games, the so-called Mean Field Game system with state constraints. For this, we show a global semiconvavity property of the value function associated with optimal control problems with state constraints. 
\end{abstract}
\section{Introduction}
The theory of Mean Field Games (MFG)  has been developed simultaneously by Lasry and Lions (\cite{8}, \cite{9}, \cite{10}) and by Huang, Malham\'{e} and Caines (\cite{h1}, \cite{h2}) in order to study  differential games with an infinite number of rational players in competition. The simplest MFG models lead to systems of partial differential equations involving two unknown functions: the value function $u=u(t,x)$ of the optimal control problem that a typical player seeks to solve, and the time-dependent density $m=(m(t))$ of the population of players:
\begin{equation}\label{mfg}
(MFG)\qquad
\begin{cases}
-\partial_t u +H(x,Du)=F(x,m) \\
\partial_t m-div(mD_pH(x,Du))=0\\
m(0)=m_0 \ \ \ \ u(x,T)=G(x,m(T)).
\end{cases}
\end{equation}
In the largest part of the literature, this system is 
studied  in the full space $(0,T)\times \R^n$ (or with space periodic data) assuming the time horizon $T$ to be finite. The system can also be associated with Neumann boundary conditions for the $u$-component, which corresponds to reflected dynamics for the players. 

The aim of this paper is to study  system \eqref{mfg} when players are confined in the closure of a bounded domain $\Omega\subset\R^n$---a set-up that arises naturally in applications. For instance,  MFG  appearing in macroeconomic models (the so-called heterogeneous agent models) almost always involve  constraints on the state variables. Indeed, one could even claim that state constraints play  a central role in the analysis since they explain heterogeneity in the economy: see, for instance, Huggett's model of income and wealth distribution~(\cite{abll, ahll}). It is also very natural to introduce constraints in pedestrians MFG models, although this variant of the problem has so far   been discussed just in an informal way. Here again,  constraints are important to explain the behavior of  crowds and it is largely believed that the  analysis of constrained problems should help regulating traffic: see, for instance, \cite{cfsw, cpt} on related issues. However, despite their relevance to applications, a general approach to MFG with state constraints seems to be missing up to now. To the best of our knowledge, the first reference on this subject is the  analysis of Huggett's model in \cite{ahll}. Other contributions are \cite{cc} and \cite{ccc},  on which we comment below. 

One of the main issues in the analysis of MFG models with state constraints is the interpretation of  system \eqref{mfg}. If, on the one hand, the meaning of the Hamilton-Jacobi equation associated with the underlying optimal control problem is well understood (see \cite{So1, So2} and \cite{cdl, ik}), on the other hand this is not the case for the continuity equation. This fact is due to several reasons: first, in contrast with  unconstrained problems, one cannot expect  $m(t)$ to be absolutely continuous with respect to the Lebesgue measure, in general. In fact, for Huggett's model~( \cite{ahll})  measure $m$ {\it always} develops a singular part at the boundary of $\Omega$. Moreover, solutions of Hamilton-Jacobi equations fail to  be of class $C^1$, in general. Thus, the gradient $Du(t,x)$---even when it exists---may develop discontinuities. In addition, the meaning of $Du(t,x)$, when the point $x$ belongs to the boundary of the domain, is totally unclear. For all these reasons, the interpretation of the continuity equation is problematic. 

To overcome the above difficulties,  the first two authors of this paper introduced in \cite{cc} the notion of {\em relaxed MFG equilibrium}. Such an equilibrium is not defined in terms of the MFG system, but follows instead the so-called Lagrangian formulation (see \cite{bb}, \cite{bc}, \cite{bcs}, \cite{b}, \cite{pc}, \cite{cms}). The main result of \cite{cc} is the existence of  MFG equilibria, which holds under fairly general assumptions. In the same paper,  the uniqueness of the solution is also derived  under an adapted version of the Lasry-Lions monotonicity condition~(\cite{10}). 
Once the existence of relaxed equilibria is ensured, the next issue to address should be regularity.  In \cite{cc}, the notion of solution was very general and yields a value function $u$ and a flow of measures $m=(m(t))$ which are merely continuous. However, in \cite{ccc}, we have shown how to improve the construction in \cite{cc} to obtain more regular solutions, that is, pairs $(u,m)$ such $u$ is Lipschitz  on $[0,T]\times \overline{\Omega}$ and the flow of measures $m:[0,T]\to \mathcal{P}(\overline{\Omega})$ is Lipschitz with respect to the Kantorovich-Rubinstein metric on $\mathcal{P}(\overline{\Omega})$, the space of probability measures on $\overline{\Omega}$.  

 However, \cite{cc} and \cite{ccc} leave the open problem of establishing a suitable sense in which the MFG system is satisfied. Such a necessity justifies the search for further regularity properties of generalized solutions. Indeed, despite its importance in nonsmooth analysis, Lipschitz regularity does not suffice to give the MFG system---in particular,   the continuity equation---a satisfactory interpretation.  
A more useful regularity property, in connection with Hamilton-Jacobi equations, is known to be semiconcavity (see, for instance, \cite{cs}). However, even for problems in the calculus of variations or optimal control, very few semiconcavity results are available in the presence of state constraints and, in fact, it was  so far known that global {\em  linear} semiconcavity should not be expected~(\cite{cm}). 

Surprisingly, in this paper we show that global semiconcavity---with a {\em fractional} modulus---does hold true. 
More precisely, denoting by  $u$ the value function of a  constrained problem in the calculus of variation (see Subsection \ref{subsec.NC}),  
we prove that 
\begin{equation*}
\frac{1}{2}u(t+\sigma,x+h)+\frac{1}{2}u(t-\sigma,x-h)-u(t,x)
\leq c(|h|+\sigma)^{\frac{3}{2}}
\end{equation*}
for allt $(t,x)\in [0,T)\times\partial\Omega$ and $h$, $\sigma$ small enough 
	(see Corollary \ref{cor1}). Actually, the above semiconcavity estimate is obtained as a corollary of a sensitivity relation (Theorem~\ref{t1}),
for the proof of which 
 key tools are provided by necessary optimality  conditions in the formulation that was introduced in \cite{ccc}.  

Using the above property,  in this paper we give---for the first time---an interpretation of  system \eqref{mfg} in the presence of state constraints, which goes as follows: 
if  $(u,m)$ is a mild solution of the constrained MFG problem (see Definition \ref{def.mildsol} below), then---as expected---$u$ is a constrained viscosity solution of 
\begin{align*}
\begin{cases}
-\partial_t u+ H(x,Du)=F(x,m(t)) \ \ \ \ \text{in} \ (0,T)\times \overline{\Omega},\\
u(x,T)=G(x,m(T)) \ \ \ \text{in} \ \overline{\Omega},
\end{cases}
\end{align*}
(in the sense of \cite{So1, So2}). Moreover---and this is our main result---there exists a  bounded continuous   vector field $\v:(0,T)\times\overline{\Omega}\rightarrow \R^n$ such that $m$ satisfies the continuity equation
	\begin{align}
	\begin{cases}
	\partial_t m+ div(\v \ m)=0, \ \ &\mbox{in} \ (0,T)\times \overline{\Omega},\\
	m(0,x)=m_0(x), \ \ \ &\mbox{in} \ \overline{\Omega}
	\end{cases}
	\end{align}
 in the sense of distributions. The vector field $\v$ is  related to $u$ in the following way: on the one hand, at any point $(t,x)$ such that $x$ is an interior point belonging to the support of $m(t)$, 
$u$ is differentiable  and  $$V(t,x)=-D_pH(x, Du(t,x)).$$ 
On the other hand, if $x$ is a boundary point on the support of $m(t)$, then one has that
\begin{equation*}
V(t,x)=-D_pH\Bigl(x,\pttx+\lambda_+(t,x)\nu(x)\Bigr),
\end{equation*}
where $\pttx$ is the tangential  component of all elements of the superdifferential of $u$ and   $\lambda_+(t,x)$ is the unique real number $\lambda$ for which $-D_pH(x, \pttx+\lambda \nu(x))$ is tangential to $\Omega$ at $x$ (see Remark~\ref{rem.jeazl}). We also prove that $u$ has time derivative at $(t,x)$ and  $\pttx+\lambda_+(t,x)\nu(x)$ can be interpreted as the correct space derivative of $u$ at $(t,x)$. For instance, we show that the Hamilton-Jacobi equation holds with an equality at any such point, that is,
$$
-\partial_t u(t,x)+ H(x,\pttx+\lambda_+(t,x)\nu(x))=F(x,m(t)),
$$
as is  the case for points of differentiability of the solution in the interior. The continuity of the vector field $\v$ is directly related to the semiconcavity of $u$. Such a rigidity  result  is reminiscent of the reformulation of the notion of viscosity solution of Hamilton-Jacobi equation with state-constraints in terms of flux-limited solutions, as described in the recent papers by Imbert and Monneau~\cite{IM} and   Guerand~\cite{Gu}. 

This paper is organized as follows. In Section 2, we introduce the notation and some preliminary results. In Section 3, under suitable assumptions, we deduce the local fractional semiconcavity of the value function associated to a variational problem with state constraints. In Section 4, we apply our semiconcavity result to constrained MFG problems. In particular, we give a new interpretation of the MFG system in the presence of state constraints. Finally, in the Appendix, we prove a technical result on directional derivatives.

\section{Preliminaries}
Throughout this paper we denote by $|\cdot|$ and $\langle  \cdot  \rangle$ , respectively, the Euclidean norm and scalar product in $\mathbb{R}^n$. We denote by $B_R$ the ball of radius $R>0$ and center $0$. Let $A\in\mathbb{R}^{n\times n}$ be a matrix. We denote by $||\cdot||$ the norm of $A$ defined as follows
$$
||A||=\max_{|x|=1} |Ax| \ \ \ \ (x\in \mathbb{R}^n).
$$
For any subset $S \subset \mathbb{R}^n$, $\overline{S}$ stands for its closure, $\partial S$ for its boundary, and $S^c$ for $\mathbb{R}^n\setminus S$. We denote by $\mathbf{1}_{S}:\mathbb{R}^n\rightarrow \{0,1\}$ the characteristic function of $S$, i.e.,
\begin{align*}
\mathbf{1}_{S}(x)=
\begin{cases}
1  \ \ \ &x\in S,\\
0 &x\in S^c.
\end{cases}
\end{align*}
We write $AC(0,T;\mathbb{R}^n)$ for the space of all absolutely continuous $\mathbb{R}^n$-valued functions on $[0,T]$, equipped with the uniform norm $||\gamma||_\infty ={\rm sup}_{[0,T]}\ |\gamma(t)|$. We observe that $AC(0,T;\mathbb{R}^n)$ is not a Banach space.\\
Let $U$ be an open subset of $\mathbb{R}^n$. $C(U)$ is the space of all continuous functions on $U$ and $C_b(U)$ is the space of all bounded continuous functions on $U$. $C^k(U)$ is the space of all functions $\phi:U\rightarrow\mathbb{R}$ that are k-times continuously differentiable. Let $\phi\in C^1(U)$. The gradient vector of $\phi$ is denoted by $D\phi=(D_{x_1}\phi,\cdots , D_{x_n}\phi)$, where $D_{x_i}\phi =\frac{\partial \phi}{\partial x_i}$. Let $\phi \in C^k(U)$ and let $\alpha=(\alpha_1,\cdots,\alpha_n) \in \mathbb{N}^n$ be a multiindex. We define $D^\alpha \phi=D^{\alpha_1}_{x_1}\cdots D^{\alpha_n}_{x_n}\phi$.
$C^k_b(U)$ is the space of all function $\phi\in C^k(U)$ and such that
$$
\|\phi\|_{k,\infty}:=\sup_{\tiny\begin{array}{c}
	x\in U\\ 
	|\alpha|\leq k
	\end{array}} |D^\alpha\phi(x)|<\infty
$$
Throughout the paper,  $\Omega$ is a bounded open subset of $\mathbb{R}^n$ with $C^2$ boundary. $C^{1,1}(\overline{\Omega})$ is the space of all the functions $C^1$ in a neighborhood $U$ of $\Omega$ and with locally Lipschitz continuous first order derivatives.\\
The distance function from  $\overline{\Omega}$ is the function $d_\Omega :\mathbb{R}^n \rightarrow [0,+ \infty[$ defined by
\begin{equation*}
d_\Omega(x):= \inf_{y \in \overline{\Omega}} |x-y| \ \ \ \ \ (x\in \mathbb{R}^n).
\end{equation*}
We define the oriented boundary distance from $\partial \Omega$ by
\begin{equation*}
\d(x)=d_\Omega(x) -d_{\Omega^c}(x) \ \ \ \ (x\in\mathbb{R}^n).
\end{equation*}
We recall that, since the boundary of $\Omega$ is of class $C^2$, there exists $\rho_0>0$ such that
\begin{equation}\label{dn}
\d(\cdot)\in C^2_b \ \ \text{on} \ \ \Sigma_{\rho_0}=\Big\{y\in B(x,\rho_0): x\in \partial\Omega\Big\}.
\end{equation}
Throughout the paper, we suppose that $\rho_0$ is fixed so that \eqref{dn} holds.\\
Let $X$ be a separable metric space. $C_b(X)$ is the space of all bounded continuous functions on $X$. We denote by $\mathscr{B}(X)$ the family of the Borel subset of $X$ and by $\mathcal{P}(X)$ the family of all Borel probability measures on $X$. The support of $\eta \in \mathcal{P}(X)$, $supp(\eta)$, is the closed set defined by
 \begin{equation*}
 supp (\eta) := \Big \{x \in X: \eta(V)>0\ \mbox{for each neighborhood V of $x$}\Big\}.
 \end{equation*}
 We say that a sequence $(\eta_i)\subset \mathcal{P}(X)$ is narrowly convergent to $\eta \in \mathcal{P}(X)$ if
 \begin{equation*}
 \lim_{i\rightarrow \infty} \int_X f(x)\,d\eta_i(x)=\int_X f(x) \,d\eta \ \ \ \ \forall f \in C_b(X).
 \end{equation*}
 We denote by $d_1$ the Kantorovich-Rubinstein distance on $X$, which---when $X$ is compact---can be characterized as follows 
 \begin{equation}\label{2.7}
 d_1(m,m')=sup\Big\{\int_X f(x)\,dm(x)-\int_X f(x)\,dm'(x)\ \Big|\ f:X\rightarrow\mathbb{R}\ \ \mbox{is 1-Lipschitz} \Big\}, 
 \end{equation}
 for all $m, m'\in\mathcal{P}(X)$.\\
 We write $\L(0,T;\mathcal{P}(\overline{\Omega}))$ for the space of all maps $m:[0,T]\rightarrow \mathcal{P}(\overline{\Omega})$ that are Lipschitz continuous with respect to $d_1$, i.e.,
 \begin{equation}\label{lipm}
 d_1(m(t),m(s))\leq C|t-s|, \ \ \ \ \forall t,s\in[0,T],
 \end{equation}
 for some constant $C\geq 0$. We denote by $\L(m)$ the smallest constant that verifies \eqref{lipm}.
 \subsection{ Semiconcave functions and generalized gradients}
 \begin{definition}
 	We say that $\omega:\R_+\rightarrow \R_+$ is a modulus if it is a nondecreasing upper semicontinuous function such that $\displaystyle\lim_{r\rightarrow 0^+} \omega(r)=0$.
 \end{definition}
 \begin{definition}
Let $\omega:\R_+\rightarrow \R_+$ be a modulus.
We say that a function $u: \overline{\Omega}\rightarrow \R$ is semiconcave with modulus $\omega$ if 
\begin{align}\label{semicon}
 	\lambda u(x)+(1-\lambda)u(y)-u(\lambda x+(1-\lambda)y)\leq \lambda(1-\lambda)|x-y| \omega(|x-y|)
 	\end{align}
 	for any pair $x$, $y\in \overline{\Omega}$, such that the segment $[x,y]$ is contained in $\overline{\Omega}$ and for any $\lambda\in[0,1]$. We call $\omega$ a modulus of semiconcavity for $u$ in $\overline{\Omega}$.
 \end{definition}
 \noindent
 A function $u$ is called semiconvex in $\overline{\Omega}$ if $-u$ is semiconcave.\\
 When the right-side of \eqref{semicon} is replaced by a term of form $C|x-y|^2$ we say that $u$ is semiconcave with linear modulus.\\
 For any $x\in \overline{\Omega}$, the sets
 \begin{eqnarray}
 D^{-}u(x) &=& \Big\{ p\in \mathbb{R}^n: \liminf_{\tiny\begin{array}{c}
 	y\rightarrow x\\
 	y\in\overline{\Omega}
 	\end{array} } \frac{u(y)-u(x)-\langle p,y-x\rangle}{|y-x|}\geq 0\Big\}\\
 D^{+}u(x) &=& \Big\{p\in \mathbb{R}^n:\limsup_{\tiny\begin{array}{c}
 	y\rightarrow x\\
 	y\in\overline{\Omega}
 	\end{array} } \frac{u(y)-u(x)-\langle p,y-x\rangle}{|y-x|}\leq 0\Big\}
 \end{eqnarray}
 are called, respectively, the (Fr\'{e}chet) subdifferential and superdifferential of $u$ at $x$.\\
 We note that if $x\in \Omega$ then, $D^+u(x)$, $D^-u(x)$ are both nonempty if and only if $u$ is differentiable in $x$. In this case we have that 
 $$
 D^+u(x)=D^-u(x)=\{Du(x)\}.
 $$
 \begin{proposition}\label{p5}
Let $u$ be a real-valued function defined on $\overline{\Omega}$. Let $x\in\partial \Omega$ and let $\nu(x)$ be the outward unit normal vector to $\partial \Omega$ in $x$. If $p \in D^+u(x)$, $\lambda\leq 0$ then $p+\lambda\nu(x)$ belongs to $D^+u(x)$ for all $\lambda\leq 0$.
 \end{proposition}
 \begin{proof}
 	Let $x\in \partial \Omega$ and let $\nu(x)$ be the outward unit normal vector to $\partial \Omega$ in $x$. Let $p \in D^+u(x)$. Let us take $\lambda\leq0$ and $y\in \overline{\Omega}$. Since $p \in D^+u(x)$ and $\lambda\leq0$, one has that
 	\begin{align*}
 	u(y)-u(x)-\langle p+\lambda\nu(x), y-x\rangle= u(y)-u(x)-\langle p, y-x\rangle-\lambda \langle \nu(x),y-x\rangle\leq o(|y-x|).
 	\end{align*}
 	Hence, $p+\lambda\nu(x)$ belongs to $D^+u(x)$.
 \end{proof}
 \noindent
 $D^+u(x)$, $D^-u(x)$ can be described in terms of test functions as shown in the next lemma.
 \begin{proposition}\label{17}
 	Let $u\in C(\overline{\Omega})$, $p\in \R^n$, and $x\in \overline{\Omega}$. Then the following properties are equivalent:
 	\begin{enumerate}
 		\item[(a)] $p \in D^+ u(x)$ (resp. $p\in D^-u(x)$);
 		\item[(b)] $p=D\phi(x)$ for some function $\phi\in C^1(\R^n)$ touching $u$ from above (resp. below);
 		\item[(c)] $p=D\phi(x)$ for some function $\phi\in C^1(\R^n)$ such that $f-\phi$ attains a local maximum (resp. minimum) at $x$.
 	\end{enumerate}
 \end{proposition}
 \noindent
 In the proof of Proposition \ref{17} it is possible to follow the same method of \cite[Proposition 3.1.7]{cs}. The following statements are straightforward extensions to the constrained case of classical results: we refer again to \cite{cs} for a proof.
 \begin{proposition}\label{pp1}
 	Let $u:\overline{\Omega}\rightarrow \R$ be semiconcave with modulus $\omega$ and let $x\in\overline{\Omega}$. Then, a vector $p\in\R^n$ belongs to $D^+u(x)$ if and only if 
 	\begin{equation}\label{pp}
 	u(y)-u(x)-\langle p,y-x\rangle\leq |y-x|\omega(|y-x|)
 	\end{equation} 
 	for any point $y\in\overline{\Omega}$ such that $[y,x]\subset \overline{\Omega}$.
 \end{proposition}
A direct consequence of Proposition \ref{pp1} is the following result.
 \begin{proposition}\label{pp2}
 	Let $u:\overline{\Omega}\rightarrow \R$ be a semiconcave function with modulus $\omega$ and let $x \in \overline{\Omega}$. Let $\{x_k\}\subset \overline{\Omega}$ be a sequence converging to $x$ and let $p_k\in D^+u(x_k)$. If $p_k$ converges to a vector $p\in\R^n$, then $p\in D^+u(x)$.
 \end{proposition}
 \begin{remark}
 	If the function $u$ depends on $(t,x)\in (0,T)\times \overline{\Omega}$, for some $T>0$, it is natural to consider the generalized partial differentials with respect to $x$ as follows
 	\begin{equation*}
 	D^+_x u(t,x):=\left\{\eta\in\mathbb{R}^n: \limsup_{h\rightarrow 0} \frac{u(t, x + h) - u(t, x) - \langle \eta,h\rangle}{h}\leq 0\right\}.
 	\end{equation*}
 \end{remark}
 \subsubsection{Directional derivatives}\label{derivd} 
 Let $\Omega$ be a bounded open subset of $\R^n$ with $C^2$ boundary. Let us first recall the definition of contingent cone.
 \begin{definition}
 	Let $x\in \overline{\Omega}$ be given. The contingent cone (or Bouligand's tangent cone) to $\Omega$ at $x$ is the set
 	\begin{equation*}
 	T_{\overline{\Omega}}(x)=\Big\{\lim_{i\rightarrow \infty} \frac{x_i-x}{t_i}: x_i\in \Omega, x_i\rightarrow x, \ t_i\in \R_+, \ t_i\downarrow 0\Big\}.
 	\end{equation*}
 \end{definition}
 \begin{remark}
 	Since $\Omega$ is a bounded open subset of $\R^n$ with $C^2$ boundary, then
 	\begin{align*}
 	\mbox{if} \ \ \ x\in \Omega \ \ &\Rightarrow \ \ T_{\overline{\Omega}}(x)=\R^n,\\
 	\mbox{if} \ \ \ x\in\partial \Omega &\Rightarrow \ \ T_{\overline{\Omega}}(x)=\Big\{\theta\in \R^n: \langle \theta,\nu(x)\rangle \leq 0\Big\},
 	\end{align*}
 	where $\nu(x)$ is the outward unit normal vector to $\partial\Omega$ in $x$.
 \end{remark}
 \begin{definition}
 	Let $x\in \overline \Omega$ and $\theta\in T_{\overline{\Omega}}(x)$. The upper and lower Dini derivatives of $u$ at $x$ in direction $\theta$ are defined as
 	\begin{equation}
 	\partial^{\uparrow}u(x;\theta)=\limsup_{\tiny\begin{array}{c}
 		h\rightarrow 0^+\\
 		\theta'\rightarrow \theta\\
 		x+h\theta'\in \overline{\Omega}
 		\end{array}} \frac{u(x+h\theta')-u(x)}{h}	
 	\end{equation}
 	and
 	\begin{equation}
 	\partial^{\downarrow}u(x;\theta)=\liminf_{\tiny\begin{array}{c}
 		h\rightarrow 0^+\\
 		\theta'\rightarrow \theta\\
 		x+h\theta'\in \overline{\Omega}
 		\end{array}} \frac{u(x+h\theta')-u(x)}{h},	
 	\end{equation}
 	respectively.\\
 	The one-sided derivative of $u$ at $x$ in direction $\theta$ is defined as
 	\begin{equation}
 	\partial^+_\theta u(x)=\lim_{\tiny\begin{array}{c}
 		h\rightarrow 0^+\\
 		\theta'\rightarrow \theta\\
 		x+h\theta'\in \overline{\Omega}
 		\end{array}} \frac{u(x+h\theta')-u(x)}{h}
 	\end{equation}
 \end{definition}
\noindent
 Let $x\in \partial \Omega$ and let $\nu(x)$ be the outward unit normal vector to $\partial\Omega$ in $x$. In the next result, we show that any semiconcave function admits one-sided derivative in all $\theta$ such that $\langle \theta,\nu(x)\rangle \leq 0$.
 \begin{lemma}\label{ax5}
	Let $u:\overline{\Omega} \rightarrow \R$ be Lipschitz continuous and semiconcave with modulus $\omega$ in $\overline{\Omega}$. Let $x\in\partial\Omega$ and let $\nu(x)$ be the outward unit normal vector to $\partial\Omega$ in $x$. Then, for any $\theta\in \R^n$ such that $\langle\theta, \nu(x)\rangle\leq0$ one has that
 	\begin{equation}\label{a25}
 	\partial^{\uparrow}u(x;\theta)=\min_{p\in D^+u(x)}\langle p,\theta\rangle=\partial^{\downarrow}u(x;\theta).
 	\end{equation}
  \end{lemma}
 \noindent
 For reader's convenience the proof is given in Appendix.
 \begin{remark}
 	We observe that Lemma \ref{ax5} also holds when $x\in \Omega$. In this case, \eqref{a25} is a direct consequence of \cite[Theorem 4.5]{7}. 
 \end{remark}
 \noindent
  Fix $x\in \partial \Omega$ and let $\nu(x)$ be the outward unit normal vector to $\partial \Omega$ in $x$. All $p\in D^+_x u(x)$ can be written as
 $$
 p=\pt+\pn
 $$
 where $\pn$ is the normal component of $p$, i.e.,
 $$
 \pn=\langle p,\nu(x)\rangle \nu(x),
 $$
 and $\pt$ is the tangential component of $p$ which satisfies 
 $$
 \langle \pt,\nu(x)\rangle=0.
 $$
 \begin{proposition}\label{prop2.3}
 	Let $x\in \partial \Omega$ and let $\nu(x)$ be the outward unit normal vector to $\partial \Omega$ in $x$. Let $u:\overline{\Omega}\rightarrow \R$ be Lipschitz continuous and semiconcave with modulus $\omega$. Then,
 	\begin{align}
 	 - \partial_{-\nu}^+u(x)=\cp(x):=\max\{\c(x):p\in D^+u(x)\},
 	\end{align}
 	where
 	\begin{equation}
 	\c(x):=\max\{\lambda \in \R: \pt+\lambda \nu(x)\in D^+u(x) \},\ \  \ \ \ \forall p\in D^+u(x).
 	\end{equation}
 \end{proposition}
 \begin{proof}
Let $x\in \partial \Omega$ and let $\nu(x)$ be the outward unit normal vector to $\partial \Omega$ in $x$. By Lemma \ref{ax5} we obtain that
 	\begin{align*}
 	&-\partial_{-\nu}^+ u(x)=-\min_{p\in D^+u(x)}\{-\langle p, \nu(x)\rangle\}=\max_{p\in D^+u(x)} \{\langle p, \nu(x)\rangle\}\\
 	&=\max \{\c(x) : p\in D^+u(x)\}=:\cp(x).
 	\end{align*}
This completes the proof.
 \end{proof}
 \subsection{Necessary conditions}\label{subsec.NC}
Let $\Omega \subset \mathbb{R}^n$ be a bounded open set with $C^2$ boundary. Let $\Gamma$ be the metric subspace of $AC(0,T;\mathbb{R}^n)$ defined by
	\begin{equation*}
	\Gamma=\Big\{\gamma\in AC(0,T;\mathbb{R}^n):\ \gamma(t)\in\overline{\Omega},\ \ \forall t\in[0,T]\Big\}.
	\end{equation*}
	For any $(t,x)\in [0,T]\times\overline{\Omega}$, we set
	\begin{equation*}
	\Gamma_t[x]=\left\{\gamma\in\Gamma: \gamma(t)=x\right\}.
	\end{equation*}
	Given $ (t,x)\in [0,T]\times\overline \Omega$, we consider the constrained minimization problem 
	\begin{equation}\label{M}
	\inf_{\gamma\in\Gamma_t[ x]}J_t[\gamma],  \ \ \ \ \mbox{where}  \ \ \ \ \
	J_t[\gamma]=\Big\{\int_{t}^T f(s,\gamma(s),\dot{\gamma}(s)) \,ds + g(\gamma(T))\Big\}.
	\end{equation}
	We denote by $\minimarc_t [x]$ the set of solutions of \eqref{M}, that is
	\begin{equation*}
	\minimarc_t[x]=\Big\{\gamma\in \Gamma_t[x]: J_t[\gamma]=\inf_{\Gamma_t[x]}J_t[\gamma]\Big\}.
	\end{equation*}
	Let $U\subset \mathbb{R}^n$ be an open set such that $\overline{\Omega}\subset U$.
	We assume that $f: [0,T]\times U\times \mathbb{R}^n\rightarrow \mathbb{R}$ and $g:U\rightarrow \mathbb{R}$ satisfy the following conditions.
	\begin{enumerate}
		\item[(g1)] $g\in C^1_b(U)$
		\item[(f0)] $f\in C\big([0,T]\times U\times \mathbb{R}^n\big)$ and for all $t\in[0,T]$ the function  $(x,v)\longmapsto f(t,x,v)$ is differentiable. Moreover, $D_xf$, $D_vf$ are continuous on $[0,T]\times U\times \mathbb{R}^n$ 
		and there exists a constant $M\geq 0$ such that
		\begin{equation}
		|f(t,x,0)|+|D_xf(t,x,0)|+|D_vf(t,x,0)|\leq M \ \ \ \ \forall\ (t,x)\in [0,T]\times U.\label{bm}
		\end{equation}
		\item[(f1)] For all $t\in[0,T]$ the map $(x,v)\longmapsto D_{v}f(t,x,v)$ is continuously differentiable and there exists a constant $\mu\geq 1$ such that
		\begin{align}
		&\frac{I}{\mu} \leq D^2_{vv}f(t,x,v)\leq I\mu,\label{f2}\\
		&||D_{vx}^2f(t,x,v)||\leq \mu(1+|v|), \label{fvx}
		\end{align}
		for all $(t,x,v)\in [0,T]\times U\times \mathbb{R}^n$, where $I$ denotes the identity matrix.
		\item[(f2)] For all $(x,v)\in U\times\mathbb{R}^n$ the function $t\longmapsto f(t,x,v)$ and the map $t\longmapsto D_vf(t,x,v)$ are Lipschitz continuous. Moreover there exists a constant $\kappa\geq 0$ such that
		\begin{align}
		&|f(t,x,v)-f(s,x,v)|\leq \kappa(1+|v|^2)|t-s|,\label{lf1}\\
		&|D_vf(t,x,v)-D_vf(s,x,v)|\leq \kappa(1+|v|)|t-s|,\label{fvt}
		\end{align} 
		for all $t$, $s\in [0,T]$, $x\in  U$, $v\in\mathbb{R}^n$.
	\end{enumerate}
\begin{remark}
By classical results in the calculus of variation (see, e.g., \cite[Theorem 11.1i]{c}), there exists at least one minimizer of \eqref{M} in $\Gamma$ for any fixed point $x\in\overline{\Omega}$.
\end{remark}
We denote by $H:[0,T]\times U\times\mathbb{R}^n \rightarrow \mathbb{R}$ the Hamiltonian 
		\begin{equation*}
		H(t,x,p)=\sup_{v\in \mathbb{R}^n} \Big\{ -\langle p,v\rangle - f(t,x,v)\Big\},\qquad \forall \ (t,x,p)\in [0,T]\times U\times \mathbb{R}^n.
		\end{equation*}
\noindent
In the next result we show the necessary conditions for our problem (for a proof see \cite{ccc}).
\begin{theorem}\label{51}
	Suppose that (g1), (f0)-(f2) hold. For any $x\in\overline{\Omega}$ and any $\gamma \in\minimarc_t[x]$ the following holds true.
	\begin{enumerate}
		\item[(i)] $\gamma$ is of class $C^{1,1}([t,T];\overline{\Omega})$.  
		\item[(ii)] There exist:
		\begin{enumerate}
			\item[(a)] a Lipschitz continuous arc $p:[t,T]\rightarrow \mathbb{R}^n$, 
			\item[(b)]a constant $\nu\in\mathbb{R}$ such that
			\begin{equation*}
			0\leq\nu\leq \max\left\{1,2\mu \ \sup_{x\in U}\Big|D_pH(T,x,Dg(x))\Big|\right\},
			\end{equation*}
		\end{enumerate}
		which satisfy the adjoint system
		\begin{align}\label{sr}
		\begin{cases}
		\dot p(s)=-D_xf(s,\gamma(s),p(s))-\Lambda(s,\gamma,p) D\d(\gamma(s)) &\mbox{for a.e.}\ s\in[t,T],\\
		p(T)= D g(\gamma(T))+ \nu D\d(\gamma(T))\mathbf{1}_{\partial\Omega}(\gamma(T))
		\end{cases}
		\end{align}
		and 
		\begin{equation}\label{cw}
	-\langle p(t),\dot{\gamma}(t)\rangle - f(t,\gamma(t),p(t))=\sup_{v\in\mathbb{R}^n} \{-\langle p(t),v\rangle-f(t,x,v)\},
		\end{equation}
where $\Lambda:[t,T]\times \Sigma_{\rho_0}\times \R^n\rightarrow \R$ is a bounded continuous function independent of $\gamma$ and $p$.
	\end{enumerate}
	Moreover, 
\begin{enumerate}
\item[(iii)] the following estimate holds
\begin{equation}\label{lstar}
||\dot{\gamma}||_\infty\leq L^\star, \ \ \ \forall \gamma\in \minimarc_t[x],
\end{equation}
where $L^\star=L^\star(\mu,M',M,\kappa,T,||Dg||_\infty,||g||_\infty)$.
\end{enumerate}
\end{theorem}
\begin{remark}
The (feedback) function $\Lambda$ in \eqref{sr} can be computed explicitly, see \cite[Remark 3.4]{ccc}.
\end{remark}
\noindent
Following the terminology of control theory, given an optimal trajectory $\gamma$, any arc $p$ satisfying \eqref{sr} and \eqref{cw} is called a \textit{dual arc} associated with $\gamma$.
\begin{remark}\label{rem.dotgamma} Following \eqref{cw} and the regularity of $H$, the derivative of the optimal trajectory $\gamma$ can be expressed in function of the dual arc:  
$$
\dot \gamma(s)= -D_pH(s,\gamma(s),p(s)), \ \ \ \ \forall  s\in[t,T].
$$
\end{remark}
\section{Sensitivity relations and fractional semiconcavity}
In this section, we investigate further the optimal control problem with state constraints introduced in Subsection \ref{subsec.NC} and show our main semiconcavity result of the value function. For this, we have to enforce the assumptions on the data. \\ 
Suppose that $f:[0,T]\times U \times \mathbb{R}^n\rightarrow\mathbb{R}$ satisfies the assumptions (f0)-(f2) and
\begin{enumerate}
	\item[(f3)] for all $s\in[0,T]$, for all $x\in U$ and for all $v$, $w \in B_R$, there exists a constant $C(R)\geq 0$ such that
			\begin{equation}
			|D_x f(s,x,v)-D_xf(s,x,w)|\leq C(R) |v-w|;
			\end{equation}
			\item[(f4)] for any $R>0$ the map $x\longmapsto f(t,x,v)$ is semiconcave with linear modulus $\omega_R$, i.e., for any $(t,v)\in[0,T]\times B_R$ one has that
			\begin{equation*}
			\lambda f(t,y,v)+(1-\lambda)f(t,x,v)-f(t,\lambda y+(1-\lambda)x,v)\leq \lambda(1-\lambda)|x-y|\omega_R(|x-y|),
			\end{equation*}
			for any pair $x$, $y\in U$ such that the segment $[x,y]$ is contained in $U$ and for any $\lambda\in[0,1]$.
\end{enumerate}
Moreover, we assume that $g:U\rightarrow \mathbb{R}$ satisfies (g1).
\noindent
Define $u:[0,T]\times \overline{\Omega}\rightarrow\mathbb{R}$ as the value function of the minimization problem \eqref{M}, i.e.,
\begin{equation}\label{vf}
u(t,x)=\inf_{\gamma\in \Gamma_t[x]} \int_{t}^T f(s,\gamma(s),\dot{\gamma}(s)) \,ds + g(\gamma(T)). 
\end{equation}
\begin{remark}
We observe that the value function $u$ is Lipschitz continuous in $[0,T]\times\overline{\Omega}$ (see \cite[Proposition 4.1]{ccc}).
\end{remark}
\noindent
Under the above assumptions on $\Omega$, $f$ and $g$ the sensitivity relations for our problem can be stated as follows.
\begin{theorem}\label{t1}
For any $\varepsilon > 0$ there exists a constant $c_\varepsilon \geq 1$ such that for any $(t,x)\in [0,T-\varepsilon]\times \overline{\Omega}$ and for any $\gamma\in\minimarc_t[x]$, denoting by $p\in\L(t,T,\R^n)$ a dual arc associated with $\gamma$, one has that 
\begin{equation*}
u(t+\sigma,x+h)-u(t,x)\leq \sigma H(t,x,p(t))+\langle p(t),h\rangle +c_\varepsilon(|h|+|\sigma|)^\frac{3}{2},
\end{equation*}
for all $h\in \R^n$ such that $x+h\in\overline{\Omega}$, and for all $\sigma\in \R$ such that $0\leq t+\sigma\leq T-\varepsilon$.
\end{theorem}
\noindent
\begin{corollario}\label{coro.ukeblsrnd}
	Let $\Omega\subset \mathbb{R}^n$ be a bounded open set with $C^2$ boundary. Let $(t,x)\in[0,T)\times \overline{\Omega}$. Let $\gamma\in \minimarc_t[x]$ and let $p\in \L(t,T;\R^n)$ be a dual arc associated with $\gamma$. Then,
	\begin{equation}
	\Big(H(s,\gamma(s),p(s)),p(s)\Big)\in D^+u(s,\gamma(s)) \ \ \ \forall \ s \in [t,T].
	\end{equation}
\end{corollario}
\noindent
A direct consequence of Theorem \ref{t1} is that $u$ is a semiconcave function. 
\begin{corollario}\label{cor1}
	Let $\Omega\subset \mathbb{R}^n$ be a bounded open set with $C^2$ boundary. The value function \eqref{vf} is locally semiconcave with modulus $\omega(r)=Cr^{\frac{1}{2}}$ in $(0,T)\times \overline{\Omega}$.
\end{corollario}
\begin{proof}
	Let $\varepsilon> 0$ and let $(t,x)\in [0,T-\varepsilon]\times\partial\Omega$. Let $\gamma \in \minimarc_t[x]$ and let $p\in\L(t,T;\R^n)$ be a dual arc assosiated with $\gamma$.
	Let $h\in \R^n$ be such that $x+h$, $x-h \in \overline{\Omega}$. Let $\sigma>0$ be such that $0\leq t-\sigma\leq t\leq t+\sigma\leq T-\varepsilon$. By Theorem \ref{t1}, there exists a constant $c_\varepsilon\geq 1$ such that
	\begin{align}\label{sgr}
	&\frac{1}{2}u(t+\sigma,x+h)+\frac{1}{2}u(t-\sigma,x-h)-u(t,x)\leq \frac{1}{2}\Big[u(t,x)+\langle p(t),h\rangle +\sigma H(t,x,p(t))\Big]\nonumber\\
	&+\frac{1}{2}\Big[u(t,x)-\langle p(t),h\rangle -\sigma H(t,x,p(t))\Big]+ c_\varepsilon(|h|+\sigma)^{\frac{3}{2}}-u(t,x)\\
	&=c_\varepsilon(|h|+\sigma)^{\frac{3}{2}}.\nonumber
	\end{align}
	Inequality \eqref{sgr} yields \eqref{semicon} for $\lambda=\frac{1}{2}$. By \cite[Theorem 2.1.10]{cs} this is enough to conclude that $u$ is semiconcave, because $u$ is continuous on $(0,T)\times \overline{\Omega}$.
\end{proof}
\subsection{Proof of Theorem \ref{t1}}
It is convenient to divide the proof of Theorem \ref{t1} in several lemmas. First, we show that $u$ is semiconcave with modulus $\omega(r)=Cr^{\frac{1}{2}}$ in $\overline{\Omega}$.
\begin{lemma}\label{teoremap}
For any $\varepsilon>0$ there exists a constant $c_\varepsilon\geq 1$ such that for any $(t,x)\in [0,T-\varepsilon]\times \overline{\Omega}$ and for any $\gamma\in \minimarc_t[x]$, denoting by $p\in \L(t,T;\R^n)$ a dual arc associated with $\gamma$, one has that
\begin{equation}\label{true}
u(t,x+h)-u(t,x)-\langle p(t),h\rangle \leq c_\varepsilon |h|^{\frac{3}{2}},
\end{equation}
for all $h\in\R^n$ such that $x+h \in \overline{\Omega}$.
\end{lemma}
\begin{proof} 
Let $\varepsilon> 0$ and let $(t,x)\in [0,T-\varepsilon]\times\overline{\Omega}$.
Let $\gamma \in \minimarc_t[x]$ and let $p\in\L(t,T;\R^n)$ be a dual arc associated with $\gamma$. 
	Let $h\in\mathbb{R}^n$ be such that $x+h\in\overline{\Omega}$. 
Let $r\in (0,\varepsilon/2]$. We denote by $\gamma_h$ the trajectory defined by
	\begin{align*}
	\gamma_h(s)= \gamma(s)+ \Big (1+\frac{t-s}{r}\Big)_+h, \ \ \ \ s\in[t,T].
	\end{align*}
We observe that, if $|h|$ is small enough, then $d_\Omega(\gamma_h(s))\leq \rho_0$ for all $s\in[t,t+r]$, where $\rho_0$ is defined in \eqref{dn}. Indeed,
	\begin{equation*}
	d_\Omega(\gamma_h(s))\leq |\gamma_h(s)-\gamma(s)|\leq \Big|\Big (1+\frac{t-s}{r}\Big)_+h \Big|\leq | h|.
	\end{equation*}
	Thus, we have that $d_\Omega(\gamma_h(s))\leq \rho_0$ for all $s\in [t,T]$ and for $| h|\leq \rho_0$.
	Denote by $\widehat{\gamma}_h$ the projection of $\gamma_h$  on $\overline{\Omega}$, i.e.,
	\begin{equation*}
	\widehat{\gamma}_h(s)={\gamma}_h(s)-d_{\Omega}({\gamma}_h(s))D\d({\gamma}_h(s)) \ \ \ \  \forall \ s \in[t,T].
	\end{equation*}
	By construction $\widehat{\gamma}_h\in AC(0,T;\R^n)$ and for $s=t$ one has that $\widehat{\gamma}_h(t)=x+h$. Moreover,
	\begin{equation}\label{s21}
	|\widehat{\gamma}_h(s)-\gamma(s)|\leq 2 |h|, \ \ \ \  \forall s\in[t,T].
	\end{equation}
	Indeed,
	\begin{align*}
	\big|\widehat{\gamma}_h(s)-\gamma(s)\big|&=\Big|{\gamma}_h(s)-d_{\Omega}({\gamma}_h(s))D \d({\gamma}_h(s))-\gamma(s)\Big|\leq |h| + d_{\Omega}({\gamma}_h(s))\\
	& \leq |h| + |{\gamma}_h(s)-\gamma(s)|\leq 2|h|,
	\end{align*}
	for all $s \in [t,T]$.	
	Furthermore, recalling \cite[Lemma 3.1]{cc}, we have that
	\begin{align}\label{d2}
	\dot{\widehat{\gamma}}_h(s)=&\dot{\gamma}(s) -\frac{h}{r}-\big\langle D\d({\gamma}_h(s)),\dot{\gamma}(s)-\frac{h}{r}\big\rangle D\d({\gamma}_h(s))\mathbf{1}_{\Omega^c}(\gamma_h(s))\\
	&-d_{\Omega}({\gamma}_h(s))D^2\d({\gamma}_h(s))\Big(\dot{\gamma}(s)-\frac{h}{r}\Big)\nonumber,
	\end{align}	
	for a.e. $s\in[t,t+r]$.
	Since $\gamma$ is an optimal trajectory for $u$ at $(t,x)$, by the dynamic programming principle, and by the definition of $\widehat{\gamma}_h$ we have that 
	\begin{align}\label{ke}
	&u(t,x+h)-u(t,x)-\langle p(t), h \rangle \leq \int_t^{t+r} f(s,\widehat\gamma_h(s),\dot{\widehat\gamma}_h(s))\,ds+ u(t+r,\overbrace{\widehat{\gamma}_h(t+r)}^{=\gamma(t+r)})\nonumber\\
	&-\int_t^{t+r}f(s,\gamma(s),\dot{\gamma}(s)) \,ds-u(t+r,\gamma(t+r))-\langle p(t),h\rangle\\
	&=\int_t^{t+r} \Big [f(s,\widehat\gamma_h(s),\dot{\widehat\gamma}_h(s))-f(s,\gamma(s),\dot{\gamma}(s))\Big ] \,ds -\langle p(t),h\rangle. \nonumber
	\end{align}
	Integrating by parts, $\langle p(t),h\rangle$ can be rewritten as
	\begin{align*}
	&-\langle p(t),h\rangle= -\big\langle p(t+r),\overbrace{\widehat{\gamma}_h(t+r)-\gamma(t+r)}^{=0}\big\rangle +\int_t^{t+r} \frac{d}{ds}\Big[\langle p(s),\widehat{\gamma}_h(s)-\gamma(s)\rangle\Big]\,ds\\
	&=\int_t^{t+r}\big\langle \dot{p}(s),\widehat{\gamma}_h(s)-\gamma(s)\big\rangle\,ds+\int_t^{t+r} \big\langle p(s),\dot{\widehat{\gamma}}_h(s)-\dot\gamma(s)\big\rangle\,ds.
	\end{align*}
Recalling that $p$ satisfies \eqref{sr} and \eqref{cw}, we deduce that
	\begin{align}\label{ph}
	-\langle p(t),h\rangle=&-\int_t^{t+r} \Big[\big\langle D_xf(s,\gamma(s),\dot{\gamma}(s)),\widehat{\gamma}_h(s)-\gamma(s) \big\rangle+\Lambda(s,\gamma,p)\big\langle D\d(\gamma(s)),\widehat{\gamma}_h(s)-\gamma(s)\big\rangle\Big]\,ds\nonumber\\
	&-\int_t^{t+r} \big\langle D_vf(s,\gamma(s),\dot{\gamma}(s)),\dot{\widehat{\gamma}}_h(s)-\dot{\gamma}(s)\big\rangle\,ds.
	\end{align}
	Therefore, using \eqref{ph}, \eqref{ke} can be rewritten as
	\begin{align}\label{u3}
	&	u(t,x+h)-u(t,x)-\langle p(t), h \rangle\leq \nonumber\\
	&\int_t^{t+r} \Big[f(s,\widehat{\gamma}_h(s),\dot{\widehat{\gamma}}_h(s))-f(s,\gamma(s),\dot{\widehat{\gamma}}_h(s))-\big\langle D_xf(s,\gamma(s),\dot{\widehat{\gamma}}_h(s)),\widehat{\gamma}_h(s)-\gamma(s) \big\rangle\Big] \,ds\nonumber\\
	&+\int_t^{t+r}\Big[ f(s,\gamma(s),\dot{\widehat\gamma}_h(s))-f(s,\gamma(s),\dot{\gamma}(s)) -\big\langle D_vf(s,\gamma(s),\dot{\gamma}(s)),\dot{\widehat{\gamma}}_h(s)-\dot{\gamma}(s)\big\rangle\Big]\,ds\nonumber\\
	&+\int_t^{t+r} \big\langle D_xf(s,\gamma(s),\dot{\widehat{\gamma}}_h(s))-D_xf(s,\gamma(s),\dot{\gamma}(s)),\widehat{\gamma}_h(s)-\gamma(s)\big\rangle \,ds\\
	& -\int_t^{t+r}\Lambda(s,\gamma,p)\big\langle D\d(\gamma(s)),\widehat{\gamma}_h(s)-\gamma(s)\big\rangle\,ds.\nonumber
	\end{align}
	Using the assumptions (f1), (f3) and (f4) in \eqref{u3} we have that
	\begin{align*}
	&u(t,x+h)-u(t,x)-\langle p(t), h \rangle\leq c\int_t^{t+r} \big|\widehat{\gamma}_h(s)-\gamma(s)\big|^2\,ds +c\int_t^{t+r} \big|\dot{\widehat{\gamma}}_h(s)-\dot{\gamma}(s)\big|^2\,ds\\
	& +C(R)\int_t^{t+r} \big|\dot{\widehat{\gamma}}_h(s)-\dot{\gamma}(s))\big|\big|\widehat{\gamma}_h(s)-\gamma(s)\big|\,ds
	-\int_t^{t+r}\Lambda(s,\gamma,p)\big\langle D\d(\gamma(s)),\widehat{\gamma}_h(s)-\gamma(s)\big\rangle\,ds,
	\end{align*}
	for some constant $c\geq 0$.
By \eqref{s21} we observe that
\begin{equation}\label{s22}
\int_t^{t+r} \big|\widehat{\gamma}_h(s)-\gamma(s)\big|^2\,ds\leq 2r |h|^2.
\end{equation}
Moreover, recalling \eqref{d2} one has that
	\begin{align*}
	&\int_t^{t+r} \big|\dot{\widehat{\gamma}}_h(s)-\dot{\gamma}(s)\big|^2\,ds\leq\frac{|h|^2}{r}+\int_t^{t+r}\Big\langle D\d(\gamma_h(s)),\dot{\gamma}(s)-\frac{h}{r}\Big\rangle^2\mathbf{1}_{\Omega^c}(\gamma_h(s))\,ds\\
	&+\int_t^{t+r}2\Big\langle D\d(\gamma_h(s)),\frac{h}{r}\Big\rangle\Big\langle D\d(\gamma_h(s)),\dot{\gamma}(s)-\frac{h}{r}\Big\rangle\mathbf{1}_{\Omega^c}(\gamma_h(s))\,ds\\
	&+\int_t^{t+r}\Big[\Big|d_{\Omega}(\gamma_h(s))D^2\d(\gamma_h(s))\Big(\dot{\gamma}(s)-\frac{h}{r}\Big)\Big|^2+\ 2d_{\Omega}(\gamma_h(s))\Big\langle D^2\d(\gamma_h(s))\Big(\dot{\gamma}(s)-\frac{h}{r}\Big),\frac{h}{r}\Big\rangle\ \Big]\,ds\\
	 &+2\int_t^{t+r} d_{\Omega}(\gamma_h(s))\Big\langle D^2\d(\gamma_h(s))\Big(\dot{\gamma}(s)-\frac{h}{r}\Big),D\d(\gamma_h(s))\Big\rangle\Big\langle D\d(\gamma_h(s)),\dot{\gamma}(s)-\frac{h}{r}\Big\rangle\mathbf{1}_{\Omega^c}(\gamma_h(s))\,ds.
	\end{align*}
	By \cite[Lemma 3.1]{cc} we obtain that
	\begin{align*}
	&\int_t^{t+r}\Big[ \Big\langle D\d(\gamma_h(s)),\dot{\gamma}(s)-\frac{h}{r}\Big\rangle^2\mathbf{1}_{\Omega^c}(\gamma_h(s))+2\Big\langle D\d(\gamma_h(s)),\frac{h}{r}\Big\rangle\Big\langle D\d(\gamma_h(s)),\dot{\gamma}(s)-\frac{h}{r}\Big\rangle\mathbf{1}_{\Omega^c}(\gamma_h(s))\Big]\,ds\\
	&=\int_t^{t+r}\Big\langle D\d(\gamma_h(s)),\dot{\gamma}(s)-\frac{h}{r}\Big\rangle\Big\langle D\d(\gamma_h(s)),\dot{\gamma}(s)+\frac{h}{r}\Big\rangle\mathbf{1}_{\Omega^c}(\gamma_h(s))\,ds\\
	&=\int_t^{t+r}\frac{d}{ds}\Big[d_{\Omega}(\gamma_h(s))\Big]\big\langle D\d(\gamma_h(s)),\dot{\gamma}(s)+\frac{h}{r}\big\rangle\mathbf{1}_{\Omega^c}(\gamma_h(s))\,ds.
	\end{align*}
	Recalling that $\gamma_h(t)$, $\gamma_h(t+r)\in \overline{\Omega}$, we observe that
	\begin{equation*}
\Big\{ s\in\ [t,t+r]:\gamma_h(s)\in \overline{\Omega}^c \Big\}=	\Big\{ s\in\ (t,t+r):\gamma_h(s)\in \overline{\Omega}^c \Big\}=\bigcup_{i\in\mathbb{N}}(s_i,t_i),
	\end{equation*}
	where $(s_i,t_i)\cap(s_j,t_j)=\emptyset$ for all $i\neq j$. Hence,
	\begin{align*}
	&\int_t^{t+r}\frac{d}{ds}\Big[d_{\Omega}(\gamma_h(s))\Big]\Big\langle D\d(\gamma_h(s)),\dot{\gamma}(s)+\frac{h}{r}\Big\rangle\mathbf{1}_{\Omega^c}(\gamma_h(s))\,ds\\
	&=\sum_{i\in\mathbb{N}}\int_{s_i}^{t_i}\frac{d}{ds}\Big[d_{\Omega}(\gamma_h(s))\Big]\Big\langle D\d(\gamma_h(s)),\dot{\gamma}(s)+\frac{h}{r}\Big\rangle\,ds.
	\end{align*}
	Integrating by parts, we get
	\begin{align*}
	&\sum_{i\in\mathbb{N}}\int_{s_i}^{t_i}\frac{d}{ds}\Big[d_{\Omega}(\gamma_h(s))\Big]\Big\langle D\d(\gamma_h(s)),\dot{\gamma}(s)+\frac{h}{r}\Big\rangle\,ds=\sum_{i\in\mathbb{N}}\Big[d_{\Omega}(\gamma_h(s))\Big\langle D\d(\gamma_h(s)),\dot{\gamma}(s)+\frac{h}{r}\Big\rangle\Big]\Big|_{s_i}^{t_i}\\
	&-\sum_{i\in\mathbb{N}}\int_{s_i}^{t_i}d_{\Omega}(\gamma_h(s))\frac{d}{ds}\Big[\Big\langle D\d(\gamma_h(s)),\dot{\gamma}(s)+\frac{h}{r}\Big\rangle\Big]\,ds.
	\end{align*}
	Owing to $d_{\Omega}(\gamma_h(s_i))=d_{\Omega}(\gamma_h(t_i))=0$ for $i\in\mathbb{N}$, $d_{\Omega}(\gamma_h(t+r))=d_{\Omega}(\gamma(t+r))=0$ and $d_{\Omega}(\gamma_h(t))=d_{\Omega}(\gamma(t))=0$, one has that
	\begin{equation}
	\sum_{i\in\mathbb{N}}\Big[d_{\Omega}(\gamma_h(s))\Big\langle D\d(\gamma_h(s)),\dot{\gamma}(s)+\frac{h}{r}\Big\rangle\Big]\Big|_{s_i}^{t_i}=0.
	\end{equation}
From now on, we assume that $|h|\leq r$. Then, recalling that $\gamma\in C^{1,1}([t,T];\overline{\Omega})$, one has that 
	$$
	\frac{d}{ds}\Big[\big\langle D\d(\gamma_h(s)),\dot{\gamma}(s)+\frac{h}{r}\big\rangle\Big]\leq C,
	$$ 
	where the constant $C$ does not dependent on $h$ and $r$. Hence, we deduce that
	\begin{equation*}
	\left|\sum_{i\in\mathbb{N}}\int_{s_i}^{t_i}d_{\Omega}(\gamma_h(s))\frac{d}{ds}\Big[\Big\langle D\d(\gamma_h(s)),\dot{\gamma}(s)+\frac{h}{r}\Big\rangle\Big]\,ds\right|\leq C|h|r,
	\end{equation*}
	and so
	\begin{equation}
	\int_t^{t+r}\frac{d}{ds}\Big[d_{\Omega}(\gamma_h(s))\Big]\big\langle D\d(\gamma_h(s)),\dot{\gamma}(s)+\frac{h}{r}\big\rangle\mathbf{1}_{\Omega^c}(\gamma_h(s))\,ds\leq C|h|r.
	\end{equation}
	Moreover, we have that
	\begin{align*}
	\int_t^{t+r}\Big|d_{\Omega}(\gamma_h(s))D^2\d(\gamma_h(s))\Big(\dot{\gamma}(s)-\frac{h}{r}\Big)\Big|^2\,ds&\leq C\int_t^{t+r}\Big|d_{\Omega}(\gamma_h(s))\Big|^2\Big|\dot\gamma(s)-\frac{h}{r}\Big|^2\,ds\\
	&\leq C\left[ r|h|^2+\frac{|h|^4}{r}+|h|^3\right],
	\end{align*}
	and
	\begin{align*}
	\int_t^{t+r} d_{\Omega}(\gamma_h(s))\Big\langle D^2\d(\gamma_h(s))\Big(\dot{\gamma}(s)-\frac{h}{r}\Big),\frac{h}{r}\Big\rangle\,ds\leq C\left(|h|^2+\frac{|h|^3}{r}\right),
	\end{align*}
	for some constant $C\geq 0$ independent on $h$ and $r$.
	Since $\langle D^2\d(x),D\d(x)\rangle=0$ $\forall x\in\mathbb{R}^n$ one has that
	\begin{equation*}
	\int_t^{t+r}d_{\Omega}(\gamma_h(s))\Big\langle D^2\d(\gamma_h(s))\Big(\dot{\gamma}(s)-\frac{h}{r}\Big),D\d(\gamma_h(s))\Big\rangle\big\langle D\d(\gamma_h(s)),\dot{\gamma}(s)-\frac{h}{r}\big\rangle\mathbf{1}_{{\Omega}^c}(\gamma_h(s))\,ds=0.
	\end{equation*}
	Hence,
	\begin{equation}\label{ss}
	\int_t^{t+r} \big|\dot{\widehat{\gamma}}_h(s)-\dot{\gamma}(s)\big|^2\,ds\leq c\left[\frac{|h|^2}{r}+ r|h|^2 +|h|^2 +|h|r+\frac{|h|^4}{r}+|h|^3+\frac{|h|^3}{r}\right].
	\end{equation}
	Moreover, using Young's inequality, (\ref{ss}) and (\ref{s22}), we deduce that
	\begin{align}\label{c2}
	&\int_t^{t+r} \big|\dot{\widehat{\gamma}}_h(s)-\dot{\gamma}(s))\big|\big|\widehat{\gamma}_h(s)-\gamma(s)\big|\,ds\leq
	\frac{1}{2}\int_t^{t+r}\big|\dot{\widehat{\gamma}}_h(s)-\dot{\gamma}(s))\big|^2\,ds + \frac{1}{2}\int_t^{t+r}\big|\widehat{\gamma}_h(s)-\gamma(s)\big|^2\,ds\nonumber\\
	&\leq \frac{1}{2}c \Big(\frac{|h|^2}{r}+ r|h|^2 +|h|^2 +|h|r+\frac{|h|^4}{r}+|h|^3 +\frac{|h|^3}{r}\Big),
	\end{align}
	where $c$ is a constant independent of $h$ and $r$. Moreover, since $$
	\int_t^{t+r}\Lambda(s,\gamma,p)\big\langle D\d(\gamma(s)),\widehat{\gamma}_h(s)-\gamma(s)\big\rangle\,ds\leq r|h|,
	$$
	and using \eqref{ss} and \eqref{c2} we have that 
	\begin{equation}\label{c1}
	u(t,x+h)-u(t,x)-\langle p(t), h \rangle \leq c\Big(\frac{|h|^2}{r}+ r|h|^2+ |h|^2+r |h|+\frac{|h|^4}{r}+|h|^3+\frac{|h|^3}{r}\Big).
	\end{equation}
	Thus, choosing $r=|h|^{\frac{1}{2}}$ in \eqref{c1}, we conclude that \eqref{true} holds. Note that the constraint on the size of $|h|$---namely $|h|\leq \rho_0$, $|h|\leq r$ and $|h|= r^2 \leq (\varepsilon/2)^2$---depends on $\varepsilon$ but not on $(t,x)$. This constraint can be removed by changing the constant $c_\varepsilon$ if necessary. This completes the proof.
\end{proof}
\begin{lemma}\label{semiv2}
For any $\varepsilon> 0$ there exists a constant $c_\varepsilon\geq 1$ such that for all $(t,x)\in[0,T-\varepsilon]\times \overline{\Omega}$ and for all $\gamma\in \minimarc_t[x]$, denoting by $p\in \L(t,T;\R^n)$ a dual arc associated with $\gamma$, one has that 
	\begin{align*}
	u(t+\sigma,x+h)-u(t,x)\leq \langle p(t),h\rangle + \sigma H(t,x,p(t)) +c_\varepsilon(|h|+|\sigma|)^\frac{3}{2}, 
	\end{align*}
	for any $h\in \mathbb{R}^n$ such that $x+h \in \overline{\Omega}$, and for any $\sigma>0$ such that $0 \leq t+\sigma\leq T-\varepsilon$.
\end{lemma}
\begin{proof}
Let $\varepsilon > 0$ and let $(t,x)\in [0,T-\varepsilon]\times \overline{\Omega}$. Let $\sigma>0$ be such that $0\leq t\leq t+\sigma\leq T-\varepsilon$ and let $h\in\mathbb{R}^n$ be such that $x+h \in \overline{\Omega}$. Let $\gamma\in \minimarc_t[x]$ and let $p\in \L(t,T;\R^n)$ be a dual arc associated with $\gamma$. By dynamical programming principle one has that
\begin{align*}
u(t+\sigma,x+h)-u(t,x)=u(t+\sigma,x+h)-u(t+\sigma,\gamma(t+\sigma))-\int_t^{t+\sigma} f(s,\gamma(s),\dot{\gamma}(s))\,ds.
\end{align*}
By Lemma \ref{teoremap} there exists a constant $c_\varepsilon\geq 1$ such that 
\begin{equation}\label{s31}
u(t+\sigma,x+h)-u(t+\sigma,\gamma(t+\sigma))\leq \langle p(t+\sigma), x+h -\gamma(t+\sigma)\rangle +c_\varepsilon\big(\big|x+h-\gamma(t+\sigma)\big|\big)^\frac{3}{2}.
\end{equation}
By Theorem \ref{51}, we have that
\begin{equation}\label{s1}
\big|x+h-\gamma(t+\sigma)\big|\leq |h|+\big|x-\gamma(t+\sigma)\big|=|h|+\left|  \int_t^{t+\sigma}\dot{\gamma}(s)\,ds\right |\leq |h|+L^\star|\sigma|.
\end{equation}
Since $\gamma\in C^{1,1}([t,T];\overline{\Omega})$, $p\in \L(t,T;\mathbb{R}^n)$, we deduce that
\begin{align}\label{s2}
&\big\langle p(t+\sigma),x+h-\gamma(t+\sigma)\big\rangle=\langle p(t+\sigma),h\rangle +\langle p(t+\sigma),\gamma(t)-\gamma(t+\sigma)\rangle\nonumber\\
&= \langle p(t+\sigma)-p(t),h\rangle + \langle p(t),h\rangle +\int_{t+\sigma}^t \langle p(t+\sigma),\dot{\gamma}(s)\rangle\,ds\\
&\leq \L(p)|\sigma||h|+ \langle p(t),h\rangle -\int_t^{t+\sigma} \langle p(s),\dot{\gamma}(s)\rangle\,ds+\L(p)|\sigma|^2.\nonumber
\end{align}
Using \eqref{s1} and \eqref{s2} in \eqref{s31}, one has that
\begin{align}\label{r3}
u(t+\sigma,x+h)-u(t,x)&\leq \langle p(t),h\rangle -\int_t^{t+\sigma} \big[ f(s,\gamma(s),\dot{\gamma}(s))+\langle p(s),\dot{\gamma}(s)\rangle\big ]\,ds+\L(p)|\sigma||h|\nonumber\\
&+\L(p)|\sigma|^2+ c_\varepsilon(|h|+|\sigma|)^\frac{3}{2}.
\end{align}
By the definition of $H$ we have that
\begin{equation*}
 -\int_t^{t+\sigma} \big[f(s,\gamma(s),\dot{\gamma}(s))+\langle p(s),\dot{\gamma}(s)\rangle\big]\,ds=\int_t^{t+\sigma}H(s,\gamma(s),p(s))\,ds
\end{equation*}
Since $\gamma \in C^{1,1}([t,T];\overline{\Omega})$ and $p\in \L(t,T;\R^n)$, we get
\begin{eqnarray}\label{fd}
&H(s,\gamma(s),p(s))=H(t,\gamma(t),p(t))+ C( |s-t|+|\gamma(s)-\gamma(t)|+|p(s)-p(t)|)
\nonumber\\
&
\leq C|\sigma|,
\end{eqnarray}
where $C$ is a positive constant independent on $h$ and $\sigma$.\\ 
Using \eqref{fd} in \eqref{r3} we conclude that
\begin{align*}
u(t+\sigma,x+h)-u(t,x)\leq \langle p(t),h\rangle + \sigma H(t,x,p(t)) + c_\varepsilon(|h|+|\sigma|)^\frac{3}{2}.
\end{align*}
This completes the proof.
\end{proof}
\begin{lemma}\label{semiv}
For any $\varepsilon> 0$ there exists a constant $c_\varepsilon\geq 1$ such that for any $(t,x)\in [0,T-\varepsilon]\times \overline{\Omega}$ and for any $\gamma\in \minimarc_t[x]$, denoting by $p\in \L(t,T;\R^n)$ a dual arc associated with $\gamma$, one has that
\begin{equation}\label{ut}
u(t-\sigma,x+h)-u(t,x)\leq \langle p(t),h\rangle -\sigma H(t,x,p(t))+c_\varepsilon(|h|+|\sigma|)^\frac{3}{2},
\end{equation}
for any $h\in \mathbb{R}^n$ such that $x+h \in \overline{\Omega}$, and for any $\sigma>0$ such that $0 \leq t-\sigma\leq T-\varepsilon$.
\end{lemma}
\begin{proof}
Let $\varepsilon > 0$ and let $(t,x)\in [0,T-\varepsilon]\times \overline{\Omega}$. Let $\sigma>0$ be such that $0\leq t-\sigma\leq T-\sigma\leq T-\varepsilon$ and let $h\in\mathbb{R}^n$ be such that $x+h \in \overline{\Omega}$. Let $\gamma\in \minimarc_t[x]$ and let $p\in \L(t,T;\R^n)$ be a dual arc associated with $\gamma$. 
We define $\gamma_h$ and $\widehat{\gamma}_h$ as in the proof of Lemma \ref{teoremap} for $r=(\sigma+|h|)^{\frac{1}{2}}$. By  \eqref{s21} and \eqref{ss} we have, for any $s\in [t,T]$,   
\begin{equation}\label{lazezrd}
|\widehat{\gamma}_h(s)-\gamma(s)|\leq 2 |h|, \ \ \ \  
	\int_t^{t+r} \big|\dot{\widehat{\gamma}}_h(s)-\dot{\gamma}(s)\big|^2\,ds\leq c(\sigma+|h|)^{\frac{3}{2}}.
\end{equation}
We finally set 
	\begin{align*}
	\widehat{\gamma}_{h,\sigma}(s) :=  \begin{cases}
	\widehat{\gamma}_h(\sigma+s) \ \ \ \ \ &s\in[t-\sigma,T-\sigma],\\
	\widehat{\gamma}_h(T) & s\in[T-\sigma,T]
	\end{cases}
	\end{align*}
and note that $\widehat{\gamma}_{h,\sigma}(t-\sigma) = \widehat{\gamma}_h(t)= x+h$. By the dynamic programming principle we obtain
	\begin{align}
	u(t-\sigma, x+h) -u(t,x) & \leq \int_{t-\sigma}^{t} f(s,\widehat{\gamma}_{h,\sigma}(s), \dot{\widehat{\gamma}}_{h,\sigma}(s))ds+ u(t, \widehat{\gamma}_{h,\sigma}(t))-u(t,x).
	\label{ukhebsd1} 
	\end{align}
We start with the estimate of the first term on the right-hand side of \eqref{ukhebsd1}. By  using the two inequalities in \eqref{lazezrd} and the regularity of $f$,  we have
	\begin{align*}
	& \int_{t-\sigma}^t f(s,\widehat{\gamma}_{h,\sigma}(s), \dot{\widehat{\gamma}}_{h,\sigma}(s))ds =
	\int_{t}^{t+\sigma} f(s-\sigma,\widehat{\gamma}_{h}(s), \dot{\widehat{\gamma}}_{h}(s))ds 
	\notag \\
	& \qquad 
	\leq 
	\int_{t}^{t+\sigma} f(s, \gamma(s), \dot{\widehat{\gamma}}_{h}(s))ds + c\sigma(\sigma+ |h|) \notag  \\
	& \qquad
	\leq 
	\int_{t}^{t+\sigma} \Bigl( f(s, \gamma(s), \dot \gamma(s))+ \langle D_vf(s,\gamma(s),\dot \gamma(s)), \dot{\widehat{\gamma}}_{h}(s)-\dot \gamma(s)\rangle  +c |\dot{\widehat{\gamma}}_{h}(s)-\dot \gamma(s)|^2\Bigr) ds  +c\sigma(\sigma+ |h|) \notag\\
	&\qquad 
	\leq 
	\int_{t}^{t+\sigma} \Bigl( f(s, \gamma(s), \dot \gamma(s))+ \langle D_vf(s,\gamma(s),\dot \gamma(s)), \dot{\widehat{\gamma}}_{h}(s)-\dot \gamma(s)\rangle  \Bigr) ds  +c(\sigma+ |h|)^{\frac32}. \notag
	\end{align*} 
Therefore, recalling that $D_vf(s,\gamma(s),\dot \gamma(s))=-p(s)$, $p$ is uniformly Lipschitz continuous, and $ \dot{\widehat{\gamma}}_{h}$ is bounded we obtain  
	\begin{align} 
	& \int_{t-\sigma}^t f(s,\widehat{\gamma}_{h,\sigma}(s), \dot{\widehat{\gamma}}_{h,\sigma}(s))ds \label{ukhebsd2} \\
	&\qquad 
	\leq
	\int_{t}^{t+\sigma} \Bigl( f(s, \gamma(s), \dot \gamma(s))- \langle p(s), \dot{\widehat{\gamma}}_{h}(s)-\dot \gamma(s)\rangle  \Bigr) ds  +c(\sigma+ |h|)^{\frac32}. \notag\\
	& \qquad \leq 
	\int_{t}^{t+\sigma}  \Bigl(f(s, \gamma(s), \dot \gamma(s))+ \langle p(s), \dot \gamma(s)\rangle \Bigr) ds - \langle p(t), \widehat{\gamma}_{h}(t+\sigma)-(x+h)\rangle  +c(\sigma+ |h|)^{\frac32}.	 \notag
	\end{align}
On the other hand, the second term in the right-hand side of \eqref{ukhebsd1} can be estimated by using Lemma \ref{teoremap} and the first inequality in \eqref{lazezrd}:
	\begin{align} 
	 u(t, \widehat{\gamma}_{h,\sigma}(t))-u(t,x) & \leq  \langle p(t), \widehat{\gamma}_{h}(t+\sigma)- x\rangle + c|\widehat{\gamma}_{h}(t+\sigma)- x|^{\frac32} \notag \\
	 & \leq \langle p(t), \widehat{\gamma}_{h}(t+\sigma)- x\rangle + c(\sigma+|h|)^{\frac32}. \label{ukhebsd3}
	 \end{align}
	Combining \eqref{ukhebsd1}, \eqref{ukhebsd2} and \eqref{ukhebsd3}, we obtain that
	\begin{align*}
	& u(t-\sigma, x+h) -u(t,x)   \\
	 & \qquad \leq \int_t^{t+\sigma} (f(s, \gamma(s), \dot \gamma(s))+\langle p(s),  \dot \gamma(s)\rangle ) ds+\langle p(t),h\rangle+c(\sigma+|h|)^{\frac32}.
	\end{align*}
 Then \eqref{cw} and the optimality of $\gamma$ imply that
	\begin{align*}
	 u(t-\sigma, x+h) -u(t,x)   & \leq - \int_t^{t+\sigma} H(s, \gamma(s), p(s)) ds+ \langle p(t), h\rangle +c(\sigma+|h|)^{\frac32}\\
	 &\leq -\sigma H(t, x, p(t))+ \langle p(t), h\rangle +c(\sigma+|h|)^{\frac32}, 
	\end{align*}
where we used again the Lipschitz continuity of $s\to H(s, \gamma(s), p(s))$. 
	This completes the proof.
\end{proof}
\noindent
We observe that Theorem \ref{t1} is a direct consequence of Lemma \ref{semiv2} and Lemma \ref{semiv}.
\section{The Mean Field Game system: from mild to pointwise solutions}
In this section we return to  mean field games with state constraints. Our aim is to give a meaning to system \eqref{mfg}. For this, we first recall the notion of constrained MFG equilibria and mild solutions of the constrained MFG problem, as introduced in \cite{cc}. Then, we investigate further regularity properties of the value function $u$. We conclude by the interpretation of the continuity equation for $m$. 
\subsection{Assumptions}\label{ipotesi}
Let $\mathcal{P}(\overline{\Omega})$  be the set of all Borel probability measures on $\overline\Omega$ endowed with the Kantorovich-Rubinstein distance $d_1$ defined in \eqref{2.7}. Let $U$ be an open subset of $\mathbb{R}^n$ and such that $\overline{\Omega}\subset U$. Assume that $F:U\times\mathcal{P}(\overline{\Omega})\rightarrow \mathbb{R}$ and $G:U\times \mathcal{P}(\overline{\Omega})\rightarrow \mathbb{R}$ satisfy the following hypotheses.
\begin{enumerate}
	\item[(D1)] For all $x\in U$, the functions $m\longmapsto F(x,m)$ and  $m\longmapsto G(x,m)$ are Lipschitz continuous, i.e., there exists $\kappa\geq 0$ such that
	\begin{align}
	|F(x,m_1)-F(x,m_2)|+ |G(x,m_1)-G(x,m_2)| \leq \kappa d_1(m_1,m_2),\label{lf} 
	\end{align}
	for any $m_1$, $m_2 \in\mathcal{P}(\overline{\Omega})$.
	\item[(D2)] For all $m\in \mathcal{P}(\overline{\Omega})$, the functions $x\longmapsto G(x,m)$ and $x\longmapsto F(x,m)$ belong to $C^1_b(U)$. Moreover
	\begin{equation*}
	|D_xF(x,m)|+|D_xG(x,m)|\leq \kappa, \ \ \ \ \forall \ x\in U, \ \forall \ m\in \mathcal{P}(\overline{\Omega}).
	\end{equation*}
\item[(D3)] For all $m\in \mathcal{P}(\overline{\Omega})$, the function $x\longmapsto F(x,m)$ is semiconcave with linear modulus, uniformly with respect to $m$. 
\end{enumerate}
\noindent
Let $L:U\times\mathbb{R}^n\rightarrow \mathbb{R}$ be a function that satisfies the following assumptions.
\begin{enumerate}
	\item[(L0)] $L\in C^1(U\times \mathbb{R}^n)$ and there exists a constant $M\geq 0$ such that
	\begin{equation}\label{bml}
	|L(x,0)|+|D_xL(x,0)|+|D_vL(x,0)|\leq M, \ \ \ \ \forall \ x\in U.
	\end{equation}
	\item[(L1)] $D_vL$ is differentiable on $U\times\mathbb{R}^n$ and there exists a constant $\mu\geq 1$ such that
	\begin{align}
	& \frac{I}{\mu} \leq D^2_{vv}L(x,v)\leq I\mu,\label{lh1}\\
	& ||D_{vx}^2L(x,v)||\leq \mu(1+|v|),\label{c6}
	\end{align}
	for all $(x,v)\in U\times \mathbb{R}^n$.
	\item[(L2)] For all $x\in U$ and for all $v$, $w\in B_R$, there exists a constant $C(R)\geq 0$ such that
	\begin{equation}
	|D_xL(x,v)-D_xL(x,w)|\leq C(R) |v-w|.
	\end{equation}
	\item[(L3)] For any $R>0$ the map $x\longmapsto L(x,v)$ is semiconcave with linear modulus, uniformly with respect to $v\in B_R$.
\end{enumerate}
\begin{remark}
For any given  $m\in \L(0,T;\mathcal{P}(\overline{\Omega}))$, the function $f(t,x,v):=L(x,v)+F(x,m(t))$ satisfies assumptions (f0)-(f4).
\end{remark}
We denote by $H:U\times \mathbb{R}^n\rightarrow \mathbb{R}$ the Hamiltonian
\begin{equation}\label{def.HH}
H(x,p)=\sup_{v\in\mathbb{R}^n}\Big\{-\langle p,v\rangle-L(x,v)\Big\},\ \  \ \ \ \forall\ (x,p)\in U\times\mathbb{R}^n.
\end{equation}
The assumptions on $L$ imply that $H$ satisfies the following conditions.
\begin{enumerate}
	\item[(H0)] 
	$H\in C^1(U\times \mathbb{R}^n)$ and there exists a constant $M'\geq 0$ such that
	\begin{equation}
	|H(x,0)|+|D_xH(x,0)|+|D_pH(x,0)|\leq M', \ \ \ \ \forall x\in U.
	\end{equation}
\item[(H1)] $D_pH$ is differentiable on $U\times\mathbb{R}^n$ and satisfies
	\begin{align}
	&	\frac{I}{\mu}\leq D_{pp}H(x,p)\leq I\mu, \ \ \ \ \ \ \ \ \ \ \ \ \ \ \forall \ (x,p)\in U\times\mathbb{R}^n,\label{h1}\\
	&||D_{px}^2  H(x,p)||\leq C(\mu,M')(1+|p|), \ \ \ \forall \ (x,p)\in U\times \mathbb{R}^n,
	\end{align}
	where $\mu$ is the constant in (L1) and $C(\mu,M')$ depends only on $\mu$ and $M'$.
	\item[(H2)] For all $x\in U$ and for all $p$, $q\in B_R$, there exists a constant $C(R)\geq 0$ such that
	\begin{equation}
	|D_xH(x,p)-D_xH(x,q)|\leq C(R) |p-q|.
	\end{equation}
	\item[(H3)] For any $R>0$ the map $x\longmapsto H(x,p)$ is semiconvex with linear modulus, uniformly with respect to  $p\in B_R$.
\end{enumerate}
\subsection{Constrained MFG equilibria and mild solutions}
For any $t\in [0,T]$, we denote by $e_t:\Gamma\to \overline \Omega$ the evaluation map defined by 
\begin{equation}\label{et}
e_t(\gamma)= \gamma(t), \ \ \ \ \forall \gamma\in\Gamma.
\end{equation}
For any $\eta\in\mathcal{P}(\Gamma)$, we define 
\begin{equation}\label{me}
m^\eta(t)=e_t\sharp\eta \ \ \ \ \forall t\in [0,T].
\end{equation}
\noindent
For any fixed $m_0\in\mathcal{P}(\overline{\Omega})$, we denote by ${\mathcal P}_{m_0}(\Gamma)$ the set of all Borel probability measures $\eta$ on $\Gamma$ such that $e_0\sharp \eta=m_0$.
For all $\eta \in \mathcal{P}_{m_0}(\Gamma)$, we set
\begin{equation*}
J_\eta [\gamma]=\int_0^T \Big[L(\gamma(t),\dot \gamma(t))+ F(\gamma(t),m^\eta(t))\Big]\ dt + G(\gamma(T),m^\eta(T)), \ \ \ \ \  \forall \gamma\in\Gamma.
\end{equation*}
For all $x \in\overline{\Omega}$ and $\eta\in\mathcal{P}_{m_0}(\Gamma)$, we define
\begin{equation*}
\Gamma^\eta[x]=\Big\{ \gamma\in\Gamma_0[x]:J_\eta[\gamma]=\min_{\Gamma[x]} J_\eta\Big\},
\end{equation*}
where $\Gamma_0[x]=\{\gamma\in\Gamma: \gamma(0)=x\}$.
\begin{definition}
	Let $m_0\in\mathcal{P}(\overline{\Omega})$. We say that $\eta\in\mathcal{P}_{m_0}(\Gamma)$ is a contrained MFG equilibrium for $m_0$ if
	\begin{equation*}
	supp(\eta)\subseteq \bigcup_{x\in\overline{\Omega}} \Gamma^\eta[x].
	\end{equation*}
\end{definition}
\noindent
We denote by $\lo$ the set of $\eta\in \mathcal{P}_{m_0}(\Gamma)$ such that $m^\eta(t)=e_t\sharp \eta$ is Lipschitz continuous. Let $\eta\in\lo$ and fix $x\in\overline{\Omega}$. Then we have that 
\begin{equation}\label{l0}
||\dot{\gamma}||_\infty \leq L_0, \ \ \ \forall \gamma\in \Gamma^\eta[x],
\end{equation}
where $L_0=L_0(\mu,M',M,\kappa,||G||_\infty,||DG||_\infty)$ (see \cite[Proposition 4.1]{ccc}).\\
\noindent
We recall the definition of mild solution of the constrained MFG problem given in \cite{cc}.
\begin{definition}\label{def.mildsol}
	We say that $(u,m)\in  C([0,T]\times \overline{\Omega})\times C([0,T];\mathcal{P}(\overline{\Omega}))$ is a mild solution of the constrained MFG problem in $\overline{\Omega}$ if there exists a constrained MFG equilibrium $\eta\in\mathcal{P}_{m_0}(\Gamma)$ such that 
	\begin{enumerate}
		\item [(i)] $ m(t)= e_t\sharp \eta$ for all $t\in[0,T]$;
		\item[(ii)] $u$ is given by
		\begin{equation}\label{v}
		u(t,x)=  \inf_{\gamma\in \Gamma_t[x]} 
		\left\{\int_t^T \left[L(\gamma(s),\dot \gamma(s))+ F(\gamma(s), m(s))\right]\ ds + G(\gamma(T),m(T))\right\},  
		\end{equation} 
		for $(t,x)\in [0,T]\times \overline{\Omega}$.
	\end{enumerate}
\end{definition}
\begin{remark}
Suppose that (L0),(L1), (D1) and (D2) hold true. Then,
\begin{enumerate}
\item[1.]there exists at least one constrained MFG equilibrium;
\item[2.]there exists at least one mild solution $(u,m)$ of the constrained MFG problem in $\overline{\Omega}$ such that
	\begin{enumerate}
		\item[(i)] $u$ is Lipschitz continuous in $[0,T]\times\overline{\Omega}$;
		\item[(ii)] $m\in\L(0,T;\mathcal{P}(\overline{\Omega}))$ and $\L(m)\leq \ L_0$ where $L_0$ is given in \eqref{l0}.
	\end{enumerate}
\end{enumerate}
For the proof see \cite{ccc}.
\end{remark}
\noindent
A direct consequence of Corollary \ref{cor1} is the following result.
\begin{corollario}\label{c41}
Let $\Omega$ be a bounded open subset of $\mathbb{R}^n$ with $C^2$ boundary. Suppose that (L0)-(L3), (D1)-(D3) hold true. Let $(u,m)$ be a mild solution of the constrained MFG problem in $\overline{\Omega}$. Then, $u$ is locally semiconcave with modulus $\omega(r)=Cr^\frac{1}{2}$ in $[0,T)\times \overline{\Omega}$.
\end{corollario}
\subsection{The Hamilton-Jacobi-Bellman equation}
Let $\Omega$ be a bounded open subset of $\R^n$ with $C^2$ boundary. Assume that $H$, $F$ and $G$ satisfy the assumptions in Section \ref{ipotesi}. Let $m\in\L(0,T;\mathcal{P}(\overline{\Omega}))$. Consider the following equation
\begin{equation}\label{1hj}
-\partial_t u+ H(x,Du)=F(x,m(t)) \ \ \ \ \text{in} \ (0,T)\times \overline{\Omega}.
\end{equation}
We recall the definition of constrained viscosity solution.
\begin{definition}\label{d1}
	Let $u\in C((0,T)\times\overline{\Omega})$. We say that:
	\begin{enumerate}
		\item[(i)] $u$ is a viscosity supersolution of \eqref{1hj} in $(0,T)\times\overline{\Omega}$ if 
		\begin{equation*}
		-\partial_t\phi(t,x)+H(x,D\phi(t,x))\geq F(x,m(t)),
		\end{equation*} 
		for any $\phi\in C^1(\R^{n+1})$ such that $u-\phi$ has a local minimum, relative to $(0,T)\times\overline{\Omega}$, at $(t,x)\in (0,T)\times\overline{\Omega}$;
		\item[(ii)] $u$ is a viscosity subsolution of \eqref{1hj} in $(0,T)\times\Omega$ if 
		\begin{equation*}
		-\partial_t\phi(t,x)+H(x,D\phi(t,x))\leq F(x,m(t)),
		\end{equation*} 
		for any $\phi\in C^1(\R^{n+1})$ such that $u-\phi$ has a local maximum, relative to $(0,T)\times\Omega$, at $(t,x)\in (0,T)\times\Omega$;
		\item[(iii)] $u$ is constrained viscosity solution of \eqref{1hj} in $(0,T)\times\overline{\Omega}$ if it is a subsolution in $(0,T)\times\Omega$ and a supersolution in $(0,T)\times\overline{\Omega}$.
	\end{enumerate}
\end{definition}
\begin{remark}
	Owing to Proposition \ref{17}, Definition \ref{d1} can be expressed in terms of subdifferential and superdifferential, i.e.,
	\begin{align*}
	&-p_1+H(x,p_2)\leq F(x,m(t)) \ \ \ \ \forall\ (t,x)\in (0,T)\times \Omega , \ \forall \ (p_1,p_2)\in D^+u(t,x),\\
	&-p_1+H(x,p_2)\geq F(x,m(t)) \ \ \ \ \forall\ (t,x)\in (0,T)\times \overline{\Omega} , \ \forall \ (p_1,p_2)\in D^-u(t,x).
	\end{align*}
\end{remark}
\noindent
A direct consequence of the definition of mild solution is the following result.
\begin{proposition}
	Let $H$ and $F$ satisfy hypotheses $(H0)-(H3)$ and $(D1)-(D3)$, respectively.
	Let $(u,m)$ be a mild solution of the constrained MFG problem in $\overline{\Omega}$. Then, $u$ is a constrained viscosity solution of \eqref{1hj} in $(0,T)\times \overline{\Omega}$.
\end{proposition}
\begin{remark}
	Given $m\in \L(0,T;\mathcal{P}(\overline{\Omega}))$, it is known that $u$ is the unique constrained viscosity solution of \eqref{1hj} in $\overline{\Omega}$ (see \cite{cdl, So1, So2}).
\end{remark}
\noindent
From now on, we set
\begin{align}
Q_m\hspace*{-1mm}=\hspace*{-1mm}\{(t,x)\hspace*{-1mm}\in\hspace*{-0.5mm} (0,T)\times\Omega: x \hspace*{-0.5mm}\in \hspace*{-0.5mm}supp (m(t))\},\; \partial Q_m\hspace*{-1mm}=\hspace*{-1mm}\{(t,x)\hspace*{-1mm}\in (0,T)\times\partial\Omega: x \hspace*{-0.5mm}\in \hspace*{-0.5mm}supp (m(t))\}. \label{defQm}
\end{align}
We note that $Q_m\cap \partial Q_m=\emptyset$ and that $Q_m\cup\partial Q_m= supp(m)\cap ((0,T)\times \overline{\Omega})$.
\begin{theorem}\label{t3}
	Let $H$ and $F$ satisfy hypotheses $(H0)-(H3)$ and $(D1)-(D3)$, respectively. Let $(u,m)$ be a mild solution of the constrained MFG problem in $\overline{\Omega}$ and let $(t,x)\in Q_m$. Then,
	\begin{equation}\label{hj1}
	-p_1+H(x,p_2)=F(x,m(t)), \ \ \ \ \forall \ (p_1,p_2)\in D^+u(t,x).
	\end{equation}
\end{theorem}
\begin{proof}
	Let $(u,m)$ be a mild solution of the constrained MFG problem in $\overline{\Omega}$. 
	Since $u$ is a constrained viscosity solution of \eqref{1hj} in $\overline{\Omega}$, we know that
	\begin{align*}
	-p_1 + H(x,p_2)\leq F(x,m(t)) \ \ \ \forall \ (t,x) \in (0,T)\times \Omega,\ \ \forall\ (p_1,p_2)\in D^+u(t,x).
	\end{align*}
	So, it suffices to prove that the converse inequality also holds. Let us take $(t,x)\in Q_m$ and $(p_1,p_2)\in D^+u(t,x)$. Since $(t,x)\in Q_m$, then there exists  an optimal trajectory $\gamma:[0,T]\rightarrow \overline{\Omega}$ such that $\gamma(t)=x$. 
	Let $r\in \R$ be small enough and such that $0\leq t-r \leq t$. Since $(p_1,p_2)\in D^+u(t,x)$ one has that
	\begin{equation*}
	u(t-r,\gamma(t-r))-u(t,x)\leq -p_1r -\langle p_2, x-\gamma(t-r)\rangle +o(r).
	\end{equation*}
	Since
	\begin{equation*}
	x-\gamma(t-r)=\int_{t-r}^t\dot{\gamma}(s)\,ds,
	\end{equation*}
	we get
	\begin{equation}\label{ps}
	\langle p_2, x-\gamma(t-r)\rangle= \int_{t-r}^t \langle p_2,\dot{\gamma}(s)\rangle \,ds.
	\end{equation}
	By the dynamic programming principle and \eqref{ps} one has that 
	\begin{align*}
	\int_{t-r}^t \Big[L(\g(s),\dg(s))+F(\g(s),m(s))\Big]\,ds&= u(t-r,\gamma(t-r))-u(t,x)\\
	&\leq-\int_{t-r}^t \langle p_2,\dg(s)\rangle\,ds -p_1 r  +o(r). 
	\end{align*}
	By our assumptions on $L$ and $F$ and by Theorem \ref{51}, one has that
	\begin{align}\label{lfg}
	&L(\g(s),\dg(s))=L(x,\dg(t)) +o(1),\nonumber\\
	&F(\g(s),m(s))=F(x,m(t)) +o(1),\\
	&\langle p_2,\dg(s)\rangle = \langle p_2,\dg(t)\rangle +o(1),\nonumber
	\end{align}
	for all $s\in [t-r,t]$.
	Hence,
	\begin{equation*}
	-p_1 -\langle p_2,\dot{\g}(t)\rangle-L(x,\dot{\gamma}(t))\geq F(x,m(t)),
	\end{equation*}
	and so by the definition of $H$ we conclude that
	\begin{equation*}
	-p_1+H(x,p_2)=-p_1 +\sup_{v\in\R^n}\{-\langle p,v\rangle-L(x,v)\}\geq-p_1 -\langle p_2,\dot{\g}(t)\rangle-L(x,\dot{\gamma}(t)) \geq F(x,m(t)).
	\end{equation*}
	This completes the proof.
\end{proof}
\begin{proposition}\label{t6}
	Let $H$ and $F$ satisfy the hypotheses $(H0)-(H3)$ and $(D1)-(D3)$, respectively. Let $(u,m)$ be a mild solution of the constrained MFG problem in $\overline{\Omega}$ and let $(t,x)\in Q_m$. Then $u$ is differentiable at $(t,x)$.
\end{proposition}
\begin{proof}
	By Theorem \ref{t3} one has that
	\begin{equation*}
	-p_1 + H(x,p_2)=F(x,m(t)) \ \ \forall \ (t,x)\in Q_m, \ \forall \ (p_1,p_2)\in D^+u(t,x).
	\end{equation*}
	Since $H(x, \cdot)$ is strictly convex and $D^+u(t,x)$ is a convex set, the above equality implies that $D^+u(t,x)$ is a singleton. Then, owing to Corollary \ref{c41} and  \cite[Proposition 3.3.4]{cs}, $u$ is differentiable at $(t,x)$.
\end{proof}
\noindent
Let $x\in\partial \Omega$. We denote by $\ho:\partial\Omega\times\mathbb{R}^n\rightarrow \R$ the tangential Hamiltonian
\begin{equation}\label{defHtau}
\ho(x,p)=\sup_{\tiny\begin{array}{c}
	v\in\R^n\\
	\langle v,\nu(x)\rangle=0
	\end{array}}
\{-\langle p, v\rangle -L(x,v)\},
\end{equation}
where $\nu(x)$ is the outward unit normal  to $\partial\Omega$ in $x$.
\begin{theorem}\label{t4}
	Let $H$ and $F$ satisfy hypotheses $(H0)$-$(H3)$ and $(D1)$-$(D3)$, respectively. Let $(u,m)$ be a mild solution of the constrained MFG problem in $\overline{\Omega}$ and let $(t,x)\in \partial Q_m$. Then,
	\begin{equation}
	-p_1+\ho(x,p_2)=F(x,m(t)), \ \ \ \forall\ (p_1,p_2)\in D^+u(t,x).
	\end{equation}
\end{theorem}
\noindent
The technical lemma is needed for the proof of Theorem~\ref{t4}.
\begin{lemma}\label{l1}
	Let $(t,x)\in (0,T)\times \partial \Omega$ and let $\nu(x)$ be the outward unit normal  to $\partial \Omega$ in $x$. Let $v\in \R^n$ be such that $\langle v,\nu(x)\rangle=0$. Then, there exists $\widehat{\gamma}\in \Gamma_t[x]$ such that $\dot{\widehat\gamma}(t)=v$.
\end{lemma}
\begin{proof}
	Let $(t,x)\in (0,T)\times \partial\Omega$ and let $\nu(x)$ be the outward unit normal vector to $\partial \Omega$ in $x$. Let $v\in \R^n$ be such that $\langle v,\nu(x)\rangle=0$. Let $R>0$ be small enough and let $\gamma$ be the trajectory defined by
	\begin{equation*}
	\gamma(s)=x+(s-t)v,
	\end{equation*}
	for all $s$ such that $|s-t|<R$. We denote by $\widehat{\gamma}$ the projection of $\gamma$ on $\overline{\Omega}$, i.e.,
	\begin{equation*}
	\widehat{\gamma}(s)=\gamma(s)-d_\Omega(\gamma(s))D\d(\gamma(s)),
	\end{equation*}
	for all $s$ such that $|s-t|<R$. By construction, we have that $\widehat{\gamma} \in \Gamma_t[x]$. We only have to prove that $\dot{\widehat{\gamma}}(t)=v$.
	Hence, recalling that $d_\Omega(\gamma(t))=0$ one has that
	\begin{align*}
	\frac{\widehat{\gamma}(s)-x}{s-t}= v-\frac{d_\Omega(\gamma(s))D\d(\gamma(s))}{s-t}=v-\left(\frac{d_\Omega(\gamma(s))-d_\Omega(\gamma(t))}{s-t}\right)D\d(\gamma(s)).
	\end{align*}
	By \cite[Lemma 3.1]{cc}, and by the definition of $\gamma$ we have that
	\begin{align*}
	&\left|\frac{d_\Omega(\gamma(s))-d_\Omega(\gamma(t))}{s-t}\right|= \left|\fint_t^s \langle D\d(\gamma(r)),\dot{\gamma}(r)\rangle\mathbf{1}_{\Omega^c}(\gamma(r))\,dr\right|\leq \fint_t^s |\langle D\d(\gamma(r)),\dot{\gamma}(r)\rangle|\,dr.
	\end{align*}
	Since $ r\longmapsto\langle  D\d(\gamma(r)),\dot \gamma(r)\rangle$ is continuous and vanishes at $r=0$,
	one has that
	$$
	\fint_t^s |\langle D\d(\gamma(r)),\dot \gamma(r)\rangle|\,dr\rightarrow 0.
	$$
	Hence,
	\begin{equation*}
	\left|\frac{d_\Omega(\gamma(s))-d_\Omega(\gamma(t))}{s-t}\right|\rightarrow 0,
	\end{equation*}
	and so $\dot{\widehat{\gamma}}(t)=v$. This completes the proof.
\end{proof}
\begin{proof}[Proof of  Theorem \ref{t4}]
	Let $(u,m)$ be a mild solution of the constrained MFG problem in $\overline{\Omega}$.
	Let us take $(t,x)\in \partial Q_m$ and $(p_1,p_2)\in D^+u(t,x)$. Let $\nu(x)$ be the outward unit normal  to $\partial \Omega$ in $x$. Let $v\in \R^n$ be such that $\langle v,\nu(x)\rangle=0$. Let $r>0$ be small enough and such that $0<t<t+r<T$. By Lemma \ref{l1} there exists $\gamma\in \Gamma_t[x]$ such that $\dot{\gamma}(t)=v$. Since $(p_1,p_2)\in D^+u(t,x)$ one has that
	\begin{equation}\label{a0}
	u(t+r,\gamma(t+r))-u(t,x)\leq \langle p_2, \gamma(t+r)-x\rangle + r p_1 + o(r).
	\end{equation}
	The dynamic programming principle ensures that
	\begin{equation}\label{a1}
	u(t+r,\gamma(t+r))-u(t,x)\geq -\int_t^{t+r} \Big[L(\gamma(s),\dot{\gamma}(s))+F(\gamma(s),m(s))\Big]\,ds.
	\end{equation}
	Moreover,
	\begin{equation}\label{a2}
	\langle p_2, \gamma(t+r)-x\rangle=\int_t^{t+r} \langle p_2,\dot{\gamma}(s)\rangle\,ds.
	\end{equation}
	Using \eqref{a1} and \eqref{a2} in \eqref{a0}, we deduce that
	\begin{align*}
	-\int_t^{t+r} \Big[L(\gamma(s),\dot{\gamma}(s))+F(\gamma(s),m(s))+\langle p_2,\dot{\gamma}(s)\rangle\Big]\,ds-r p_1\leq o(r).
	\end{align*}
	By our assumptions on $L$ and $F$ and by Theorem \ref{51}, one has that
	\begin{align}\label{lfg1}
	&L(\g(s),\dg(s))=L(x,\dg(t)) +o(1),\nonumber\\
	&F(\g(s),m(s))=F(x,m(t)) +o(1),\\
	&\langle p_2,\dg(s)\rangle = \langle p_2,\dg(t)\rangle +o(1),\nonumber
	\end{align}
	for all $s\in [t,t+r]$.
	Using \eqref{lfg1}, dividing by $r$, and passing to the limit for $r \rightarrow 0$ we obtain 
	\begin{equation}\label{a3}
	-p_1 -\langle p_2,v\rangle-L(x,v)-F(x,m(t))\leq 0.
	\end{equation}
	By the arbitrariness of $v$ and  the definition of $\ho$, \eqref{a3} implies that 
	\begin{equation*}
	-p_1+ \ho(x,p_2)\leq F(x,m(t)).
	\end{equation*}
	Now, we prove that the converse inequality also holds. Let $\gamma:[0,T]\rightarrow \overline{\Omega}$ be an optimal trajectory  such that $\gamma(t)=x$. Since $\gamma(t)\in \partial{\Omega}$, and $\gamma(s)\in \overline{\Omega}$ for all $s\in [0,T]$ one has that $\langle\dot{\gamma}(t),\nu(x)\rangle=0$. Let $r>0$ be small enough and such that $0<t-r\leq t$. Since $(p_1,p_2)\in D^+u(t,x)$, and by the dynamic programming principle one has that
	\begin{align*}
	&\int_{t-r}^t \Big[L(\gamma(s),\dot{\gamma}(s))+F(\gamma(s),m(s))\Big]\,ds= u(t-r,\gamma(t-r))-u(t,\gamma(t))\\
	&\leq -\langle p_2,\gamma(t)-\gamma(t-r)\rangle -r p_1 + o(r).
	\end{align*}
	Hence, we obtain 
	\begin{align*}
	\int_{t-r}^t \Big[L(\gamma(s),\dot{\gamma}(s))+F(\gamma(s),m(s))+\langle p_2,\dot{\gamma}(s)\rangle\Big]\,ds +r p_1\leq o(r).
	\end{align*}
	Arguing as above we deduce that
	\begin{align*}
	-p_1 -[\langle p_2,\dot{\gamma}(t) \rangle +L(x,\dot{\gamma}(t))]\geq F(x,m(t)).
	\end{align*}
	Since $\langle\dot{\gamma}(t),\nu(x)\rangle=0$,  by the definition of $\ho$ we conclude that
	\begin{align*}
	-p_1+\ho(x,p_2)&=-p_1 + \sup_{\tiny\begin{array}{c}
		v\in\R^n\\
		\langle v,\nu(x)\rangle=0
		\end{array}}\{-\langle p_2,v \rangle -L(x,v)\}\\
	&\geq-p_1-\langle p_2,\dot{\gamma}(t) \rangle -L(x,\dot{\gamma}(t))\geq F(x,m(t)).
	\end{align*}
	This completes the proof.
\end{proof}
\noindent

\noindent
\begin{remark}\label{r2}
	Let $(t,x) \in \partial Q_m$. By the definition of $\ho$ for all $p\in D^+_x u(t,x)$ one has that
	\begin{equation*}
	\ho(x,p)=\ho(x,\pt),
	\end{equation*}
	where $\pt$ is the tangential component of $p$.
\end{remark}
\noindent
In the next result, we give a full description of $D^+u(t,x)$ at $(t,x)\in \partial Q_m$.
\begin{proposition}\label{rmax}
	Let $(u,m)$ be a mild solution of the constrained MFG problem in $\overline{\Omega}$ and let $(t,x)\in \partial Q_m$. The following holds true.
	\begin{enumerate}
		\item[(a)] The partial derivative of $u$ with respect to $t$, denoted by $\partial_t u(t,x)$, does exist and
		\begin{equation*}
		D^+u(t,x)=\{\partial_t u(t,x)\}\times D^+_xu(t,x).
		\end{equation*}
		\item[(b)] All $p\in D_x^+u(t,x)$ have the same tangential component, which will be denoted by $\pttx$, that is,
		\begin{equation}\label{pt1}
		\Big\{p^\tau \in \R^n: p\in D^+_xu(t,x)\Big\}=\big\{\pttx\big\}.
		\end{equation}
		\item[(c)] For all $\theta\in \mathbb{R}^n$ such that $|\theta|=1$ and $\langle\theta,\nu(x)\rangle=0$ one has that
		\begin{equation}\label{dt}
		\partial_\theta^+ u(t,x)=\langle \pttx,\theta \rangle.
		\end{equation}
		Moreover,
		\begin{equation}\label{dn1}
		- \partial_{-\nu}^+ u(t,x)=\lambda_+(t,x):= \max\{\lambda_p(t,x): p\in D^+_xu(t,x)\},
		\end{equation}
		where
		\begin{equation*}
		\c(t,x)=\max\{\lambda \in \R : \pttx + \lambda \nu(x) \in D_x^+u(t,x)\}, \ \ \forall p\in D_x^+u(t,x).
		\end{equation*}
		\item[(d)] $D^+_xu(t,x)=\{p\in \R^n: p=\pttx+\lambda\nu(x), \  \lambda\in(-\infty, \cp(t,x)]\}$.
	\end{enumerate}
\end{proposition}
\begin{proof}
	Let $(u,m)$ be a mild solution of the constrained MFG problem in $\overline{\Omega}$. Let $(t,x)\in \partial Q_m$ and let $\nu(x)$ be the outward unit normal  to $\partial \Omega$ in $x$. Recall that, by Theorem \ref{t4} and  Remark \ref{r2},
	\begin{equation}\label{e2}
	-p_1+\ho(x,\ptd)=F(x,m(t)), \ \ \ \forall (p_1,p_2)\in D^+u(t,x).
	\end{equation}
Let us prove $(a)$ and $(b)$ together, arguing by contradiction. Let $p=(p_1,p_2)$, $q=(q_1,q_2) \in D^+ u(t,x)$ be such that $p^\tau_2\neq q^\tau_2$. Let $\lambda\in[0,1]$. Since $D^+u(t,x)$ is a convex set, we have that $$
	\pl=(p_{1,\lambda},p_{2,\lambda})=(\lambda p_1+(1-\lambda)q_1,\lambda p_2+(1-\lambda)q_2)\in D^+u(t,x).$$ 
	Moreover,  observe that
	\begin{align*}
	\lambda (\ptd+\pnd)+(1-\lambda)(\qtd+\qnd)= [\lambda \ptd+(1-\lambda)\qtd]+ [\lambda \pnd+(1-\lambda)\qnd]= \plt+\pln.
	\end{align*}
	Since $\pl\in D^+u(t,x)$, \eqref{e2} holds true and
	\begin{align*}
	\ho(x,\plt)&=p_{1,\lambda}+F(x,m(t))=\lambda p_1+(1-\lambda)q_1+F(x,m(t))\\
	&= \lambda[p_1+F(x,m(t))]+ (1-\lambda)[q_1+F(x,m(t))].
	\end{align*}
	Since $\ho$ is strictly convex on the orthogonal complement, $(\nu(x))^\perp$, of $\nu(x)$,  recalling that $p$ and $q$ satisfy \eqref{e2}  we have that
	\begin{equation*}
	\lambda \ho(x,\ptd)+(1-\lambda)\ho(x,\qtd)>\ho(x,\plt)=\lambda \ho(x,\ptd)+(1-\lambda)\ho(x,\qtd).
	\end{equation*}
	So, we conclude that $p_1=q_1$ and $\ptd=\qtd$. Thus, $(a)$ and $(b)$ hold true.
In order to prove $(c)$, let $\theta\in \R^n$ be such that $|\theta|=1$ and $\langle \theta, \nu(x)\rangle=0$. By the local semiconcavity of $u$ in $(0,T)\times \overline{\Omega}$,  Lemma \ref{ax5}, and  $(b)$ we deduce that
	\begin{equation*}
	\partial_\theta^+ u(t,x)=\min_{p\in D^+_xu(t,x)}\langle p, \theta \rangle =\langle \pttx, \theta \rangle,
	\end{equation*}
	which proves \eqref{dt}. Appealing to Proposition \ref{prop2.3},  the local semiconcavity of $u$ implies that
	\begin{align*}
	-\partial_{-\nu}^+ u(t,x)=\max \{\c(t,x) : p\in D^+_xu(t,x)\}=:\cp(t,x),
	\end{align*}
	where
	\begin{equation*}
	\c(t,x)=\max \{\lambda\in \R:  \pttx+ \lambda \nu(x) \in D^+_xu(t,x)\}.
	\end{equation*}
Finally, Proposition \ref{p5} and $(c)$ yield $(d)$. This completes the proof.
\end{proof}
\begin{theorem}\label{uc1}
	Let $(u,m)$ be a mild solution of the constrained MFG problem in $\overline{\Omega}$. Then the following holds true.
	\begin{enumerate}
		\item[(i)] For any $(t,x)\in (0,T)\times \Omega$ one has that
		\begin{equation}
		\limsup_{\tiny\begin{array}{c}
			(s,y)\in (0,T)\times\Omega\\
			(s,y)\rightarrow (t,x)
			\end{array}}
		D^+u(s,y) \subset D^+u(t,x).
		\end{equation}
		In particular, for all $(t,x)\in Q_m$,  
		\begin{equation}
		\limsup_{\tiny\begin{array}{c}
			(s,y)\in (0,T)\times\Omega\\
			(s,y)\rightarrow (t,x)
			\end{array}} D^+u(s,y)= \Big\{\Big(\partial_tu(t,x), Du(t,x)\Big)\Big\}.
		\end{equation}
		\item[(ii)]  Let $(t,x)\in \partial Q_m$. Then,
		\begin{equation}\label{ds}
		\lim_{\tiny\begin{array}{c}
			(s,y)\in (0,T)\times \Omega,\\
			\mbox{\rm $u$ differentiable at $(s,y)$},\\
			(s,y)\rightarrow (t,x)
			\end{array}} (\partial_tu(s,y), Du(s,y))=\Big(\partial_tu(t,x),\ \pttx+\cp(t,x)\nu(x)\Big),
		\end{equation}
where $\pttx$ and $\cp(t,x)$ are given in \eqref{pt1} and \eqref{dn1}, respectively.
		Moreover, one has that
		\begin{equation}\label{kjhzfdnc}
		-\partial_t u(t,x)+ H(x, \pttx + \lambda_+(t,x) \nu(x))=F(x,m(t)).  
		\end{equation}
	\item[(iii)] Let $(t,x)\in \partial Q_m$. Then,
		\begin{equation}\label{dsbis}
		\lim_{\tiny\begin{array}{c}
			(s,y)\in \partial Q_m,\\
			(s,y)\rightarrow (t,x)
			\end{array}} (\partial_tu(s,y), D^\tau_xu(s,y)+\cp(s,y)\nu(y))=\Big(\partial_tu(t,x),\ \pttx+\cp(t,x)\nu(x)\Big),
		\end{equation}
where $\pttx$ and $\cp(t,x)$ are given in \eqref{pt1} and \eqref{dn1}, respectively.		
	\end{enumerate}
\end{theorem}
\begin{proof}
Let $(u,m)$ be a mild solution of the constrained MFG problem in $\overline{\Omega}$. By Corollary \ref{c41},  Proposition \ref{t6}, and  \cite[Proposition 3.3.4]{cs} we deduce that $(i)$ holds true. Hence, we only need to analyze   $(ii)$ and $(iii)$.\\
\underline{Step 1.} \\
Let $(t,x)\in \partial Q_m$. Let $u$ be differentiable at $(s_k,y_k)\in (0,T)\times \Omega$ with $(s_k,y_k)\to (t,x)$. Since $u$ is locally semiconcave,
the bounded sequence $(\partial_tu(s_k,y_k), Du(s_k,y_k))$  has  a subsequence (labelled in the same way) which  converges to $(p_1,p_2)\in D^+u(t,x)$. Then Proposition \ref{rmax} implies that $p_1=\partial_tu(t,x)$ and that there exists $\overline{\lambda} \in (-\infty, \lambda_+(t,x)]$ such that $p_2=\pttx +\overline{\lambda} \nu(x)$. To prove \eqref{ds}, it only remains  to show that $\overline{\lambda} = \lambda_+(t,x)$. This will be achieved in Step~3. 
Since $u$ is a viscosity solution of the Hamilton-Jacobi equation and is differentiable at $(s_k,y_k)$, we have that
\begin{equation}\label{k13}
-\partial_t u(s_k,y_k)+ H(y_k, Du(s_k,y_k))=F(y_k,m(s_k)).
\end{equation}
Passing to the limit in \eqref{k13} we obtain 
\begin{equation}\label{kjhzfdnc0}
-\partial_t u(t,x)+ H(x, \pttx +\overline{\lambda} \nu(x))=F(x,m(t)).  
\end{equation}
\underline{Step 2.} \\
The next step consists in proving that \eqref{kjhzfdnc} holds by choosing a particular sequence of points. Let $(t_k,x_k)\in (0,T)\times \Omega$ be a sequence such that:
\begin{enumerate}
	\item[1.] $(t_k,x_k)\xrightarrow{k\rightarrow \infty}(t,x)$;
	\item[2.] $u$ is differentiable in $(t_k,x_k)$;
	\item[3.] $\lim_{k\rightarrow +\infty}\frac{(t_k-t,x_k-x)}{|(t_k-t,x_k-x)|}=(0,-\nu(x)).$
\end{enumerate}
\noindent
Arguing as above, we know that any cluster point of $(Du(t_k,x_k))$ is of the form $\pttx+\tilde\lambda \nu(x)$, with $\tilde \lambda\leq \lambda_+(t,x)$, and satisfies
\begin{equation}\label{kshbfdgn}
-\partial_t u(t,x)+ H(x, \pttx +\tilde \lambda \nu(x))=F(x,m(t)).  
\end{equation}
On the other hand, by the local semiconcavity of $u$ (Theorem \ref{t1}), we also have that
$$
u(t,x)-u(t_k,x_k)-\partial_tu(t_k,x_k)(t-t_k)-\langle Du(t_k,x_k), (x-x_k)\rangle \leq C( |t-t_k|+|x-x_k|)^{3/2}. 
$$
Therefore,
$$
u(t,x)-u(t,x_k)-\partial_tu(t_k,x_k)(t-t_k)-\langle Du(t_k,x_k), (x-x_k)\rangle \leq C(( |t-t_k|+|x-x_k|)^{3/2}+|t_k-t|). 
$$
Dividing this inequality by $|(t_k-t,x_k-x)|$ and passing to the limit, we obtain 
$$
-\partial_{-\nu}^+u(t,x)- \langle \pttx+ \tilde \lambda \nu(x), \nu(x)\rangle \leq 0.
$$
By \eqref{dn1} we have that
$$
\lambda_+(t,x)= -\partial_{-\nu}^+u(t,x) \leq \tilde \lambda. 
$$
This proves that $ \tilde \lambda= \lambda_+(t,x)$, whereas \eqref{kjhzfdnc}  follows from \eqref{kshbfdgn}.\\
\underline{Step 3.}\\
We finally show that the limit point $\overline{\lambda}$, defined in Step~1, equals $\lambda_+(t,x)$. Indeed, arguing by contradiction, let us assume that $\overline{\lambda} <\lambda_+(t,x)$. Then, by \eqref{kjhzfdnc0},  \eqref{kjhzfdnc}, and the  strict convexity of $H$, we have that, for any $\lambda \in (\overline{\lambda}, \lambda_+(t,x))$, 
\begin{align*}
F(x,m(t))>&  -\partial_t u(t,x)+ H(x, \pttx + \lambda \nu(x)) \geq 
-\partial_t u(t,x)+ H^\tau(x, \pttx + \lambda \nu(x)) \\
& =
-\partial_t u(t,x)+ H^\tau (x, \pttx).
\end{align*}
By Theorem \ref{t4}, we deduce that 
$$
-\partial_t u(t,x)+ H^\tau (x, \pttx)=F(x,m(t)), 
$$
which leads to a contradiction. Therefore, we have that  $\overline{\lambda} =\lambda_+(t,x)$,  which in turn implies \eqref{ds}.\\
\underline{Step 4.} \\
The proof of point $(iii)$ runs exactly along the same lines as for point $(ii)$: if $(s_k,y_k)$ belongs to $\partial Q_m$ and converges to $(t,x)$, then the bounded sequence $(\partial_tu(s_k,y_k), D^\tau_x u(s_k,y_k)+\lambda_+(s_k,y_k)\nu(y_k))$ converges (up to a subsequence) to some $(p_1,p_2)\in  D^+u(t,x)$. As in Step~1, we have that $p_1=\partial_tu(t,x)$ while $p_2=\pttx+\overline{\lambda} \nu(x))$ for some $\overline{ \lambda} \leq \lambda_+(t,x)$ and 
$$
-\partial_tu(t,x) +H(x, \pttx+\overline{\lambda} \nu(x))=F(x,m(t)). 
$$
Then,  as in Step~3, we conclude that $\overline{\lambda}= \lambda_+(t,x)$.


\end{proof}
\noindent
A direct consequence of the results of this section is the following theorem.
\begin{theorem}\label{thm.iaezd}
	Let $H$, $F$ and $G$ satisfy hypotheses $(H0)-(H3)$ and $(D1)-(D3)$, respectively.
	Then, $u$ is a constrained viscosity solution of
	\begin{align*}
	\begin{cases}
	-\partial_t u+ H(x,Du)=F(x,m(t)) \ \ \ \ \text{in} \ (0,T)\times \overline{\Omega}\\
	u(x,T)=G(x,m(T)) \ \ \ \text{in} \ \overline{\Omega}.
	\end{cases}
	\end{align*}
	Moreover, $u$ is differentiable at any $(t,x)\in Q_m$ with
	\begin{equation*}
	-\partial_t u+ H(x,Du)=F(x,m(t)) \ \ \ \ \text{in} \ Q_m,
	\end{equation*}
	while, on $\partial Q_m$, the time-derivative $\partial_t u$ exists and satisfies the equation
		\begin{equation*}
	-\partial_t u+ H^\tau(x,\pttx)=F(x,m(t)) \ \ \ \ \text{in} \ \partial Q_m.
	\end{equation*}
\end{theorem}
\begin{corollario}\label{coro.hajebzfd} 
Let $H$, $F$ and $G$ satisfy hypotheses $(H0)-(H3)$ and $(D1)-(D3)$, respectively. Let $\eta\in {\mathcal P}_{m_0}(\Gamma)$ be a constrained MFG equilibrium and $(u,m)$ be the associated mild solution of the constrained MFG problem in $\overline{\Omega}$.  If $(t,x)\in Q_m\cup \partial Q_m$, then there exists $y\in \overline{\Omega}$ and an optimal trajectory $\gamma\in \Gamma^\eta[y]$ such that $\gamma(t)=x$. Moreover, if $p$ is the dual arc associated with $\gamma$, then 
\begin{equation}\label{dotgamma}
\dot \gamma(t)= -D_pH(x,p(t))\; {\rm where}\; p(t)=
\left\{\begin{array}{ll} Du(t,x) & {\rm if }\; (t,x)\in Q_m,\\
\pttx+ \lambda_+(t,x)\nu(x) & {\rm if }\; (t,x)\in \partial Q_m.
 \end{array}\right.
\end{equation}
\end{corollario}
\begin{proof} The existence of $\gamma$ is an easy consequence of the definition of $m$ and  the uniform Lipschitz continuity of optimal trajectories. Let us now check that \eqref{dotgamma} holds. In view of Remark \ref{rem.dotgamma}, we have 
$$
\dot \gamma(t)= -D_pH(x,p(t)) ,
$$
where, by Corollary \ref{coro.ukeblsrnd}, $(H(x,p(t))-F(x,m(t)), p(t))$ belongs to $D^+u(t,x)$.  Then Proposition \ref{rmax} implies that $p(t)= Du(t,x)$ if $(t,x)\in Q_m$, while $p(t)=\pttx+\overline{\lambda} \nu(x)$ for some $\overline{\lambda} \leq \lambda_+(t,x)$ if $(t,x)\in \partial Q_m$. 

It remains to check that, in this second case, $\overline{\lambda}= \lambda_+(t,x)$. As $\gamma$ is of class $C^{1,1}([t,T],\overline{\Omega})$ and remains in $\overline \Omega$ with $\gamma(t)=x\in \partial \Omega$, we have that $\langle \dot \gamma(t), \nu(x)\rangle=0$. In particular
$$
\frac{d}{d\lambda} \ H(x, \pttx+\lambda \nu(x))\Big |_{\lambda=\overline{\lambda}}= \langle D_pH(x, \pttx+\overline{\lambda} \nu(x)), \nu(x)\rangle=
-\langle \dot \gamma(t), \nu(x)\rangle = 0.
$$
This proves that  the strictly convex map $\lambda \mapsto H(x, \pttx(t,x)+\lambda \nu(x))$ has a (unique) minimum at $\lambda = \overline{\lambda}$.  On the other hand, by Theorem \ref{uc1} and Theorem \ref{thm.iaezd} we have that
\begin{align*}
F(x,m(t))+\partial_tu(t,x) & = H(x, \pttx+\lambda_+(t,x)\nu(x))= H^\tau (x, \pttx)\\
& = H^\tau (x, \pttx+\lambda_+(t,x)\nu(x)).
\end{align*}
So, if $\hat v\in \R^n$, with $\langle \hat v, \nu(x)\rangle=0$, is a maximum point for the envelope formula in \eqref{defHtau} which represents $H^\tau (x, \pttx+\lambda_+(t,x)\nu(x))$, then $\hat v$ is also a maximizer  of \eqref{def.HH}, which gives  $H(x, \pttx+\lambda_+(t,x)\nu(x))$. By the uniform convexity of  $H$, this fact yields
$$
\hat v = -D_pH(x, \pttx+\lambda_+(t,x)\nu(x)). 
$$
So,
$$
0= \langle \hat v, \nu(x)\rangle  =- \langle D_pH(x, \pttx+\lambda_+(t,x)\nu(x)), \nu(x)\rangle,
$$
which proves that $\lambda_+(t,x)$ also  minimizes the strictly convex map $\lambda \mapsto H(x, \pttx+\lambda \nu(t,x))$. This shows that $\overline{\lambda}= \lambda_+(t,x)$ thus completing the proof.
\end{proof}
\begin{remark}\label{rem.jeazl} From the above proof it follows that, for $(t,x)\in \partial Q_m$,  $\lambda_+(t,x)$ can be characterized as the unique $\lambda \in \R$ such that the vector $- D_pH(x,\pttx+\lambda \nu(x))$ is tangent to $\Omega$ at $x$, i.e., such that 
	$$
	\langle -D_pH(x,\pttx+\lambda \nu(x)), \nu(x)\rangle =0.
	$$ 
\end{remark}
\subsection{The continuity equation}
The main result of this section is the following theorem.
\begin{theorem}\label{t46bis}
	Let $\Omega$ be a bounded open subset of $\R^n$ with $C^2$ boundary. Let $H$ and $F$ satisfy hypotheses $(H0)-(H3)$ and $(D1)-(D3)$, respectively. Let $m_0\in \mathcal{P}(\overline{\Omega})$ and let $(u,m)$ be a mild solution of the constrained MFG problem in $\overline{\Omega}$. Then, there exists  a bounded  continuous map $\v:(0,T)\times \overline{\Omega}\rightarrow \R^n$ such that $m$ is a solution in the sense of distribution of the continuity equation
	\begin{align}\label{eccbis}
	\begin{cases}
	\partial_t m+ div(\v \ m)=0, \ \ &\mbox{in} \ (0,T)\times \overline{\Omega},\\
	m(0,x)=m_0(x), \ \ \ &\mbox{in} \ \overline{\Omega}.
	\end{cases}
	\end{align}
that is, for all $\phi\in C^1_c((0,T)\times\overline{\Omega})$ one has that
\begin{equation*}
	0=\int_0^T\int_{\overline{\Omega}} \Big[\partial_t\phi(t,x)+\langle D\phi(t,x),\v(t,x)\rangle\Big]m(t,dx)\,dt.
	\end{equation*}
	Moreover, $\v$ is given on $supp(m)$ by
	\begin{align}\label{campobis}
	\v(t,x)=
	\begin{cases}
	-D_pH\big(x,Du(t,x)\big)\ \ \ \  \ &\mbox{if} \ (t,x)\in Q_m,\\
	-D_pH\big(x,\pttx+\cp(t,x)\nu(x)\big) &\mbox{if} \ (t,x)\in \partial Q_m,
	\end{cases}
	\end{align} 
	where $Q_m$ and $\partial Q_m$ are defined in \eqref{defQm}, whereas $\pttx$ and $\cp(t,x)$ are given in \eqref{pt1} and \eqref{dn1}, respectively.
\end{theorem}
\begin{proof} Let us define $\v$ on $supp(m)$ by \eqref{campobis}. By Theorem \ref{uc1} $\v$ is continuous on the set $Q_m\cup\partial Q_m$. Since $Q_m\cup\partial Q_m$ is relatively closed in $(0,T)\times \overline{\Omega}$, using the Tietze extension theorem (\cite[Theorem 5.1]{d1}) we can extend $\v$ continuously to $(0,T)\times \overline{\Omega}$. It remains to check that \eqref{eccbis} holds. Let $\eta$ be a constrained MFG equilibrium associated with $(u,m)$. Then, by the definition of $Q_m$ and $\partial Q_m$, recalling Corollary \ref{coro.hajebzfd} we have that $(t,\gamma(t))\in Q_m\cup \partial Q_m$ and $\dot\gamma(t)= \v(t,\gamma(t))$
for any $t\in (0,T)$ and $\eta-$a.e. $\gamma\in \Gamma$. So, for any $\phi\in C^1_c((0,T)\times\overline{\Omega})$, one has that
\begin{align*}
\frac{d}{dt} \int_{\overline{\Omega}} \phi(t,x)m(t,dx) & = \frac{d}{dt} \int_{\Gamma} \phi(t,\gamma(t)))\eta(d\gamma)
 = \int_{\Gamma} (\partial_t\phi(t,\gamma(t))+ \langle D\phi(t,\gamma(t)), \dot \gamma(t)\rangle) \eta(d\gamma) \\
& = \int_{\Gamma} (\partial_t\phi(t,\gamma(t))+ \langle D\phi(t,\gamma(t)), \v(t, \gamma(t))\rangle) \eta(d\gamma) \\
& = \int_{\overline{\Omega}}  (\partial_t\phi(t,x)+ \langle D\phi(t,x), \v(t, x)\rangle m(t, dx).
\end{align*}
The conclusion follows by integrating the above identity over $[0,T]$.
\end{proof}
\section{Appendix: proof of Lemma \ref{ax5}}
\subsection{Proof of Proposition \ref{17}}
The proof of Proposition \ref{17} relies on the following  technical lemma.
\begin{lemma}
	Let $w:(0,+\infty)\rightarrow [0,+\infty)$ be an upper semicontinuous function such that $\lim_{r\rightarrow 0} w(r)=0$. Then there exists a continuous nondecreasing function $w_1:[0,+\infty)\rightarrow [0,+\infty)$ such that
	\begin{enumerate}
		\item[(i)] $w_1(r)\rightarrow 0$ as $r\rightarrow 0$,
		\item[(ii)] $w(r)\leq w_1(r)$ for any $r\geq 0$,
		\item[(iii)] the function $\xi(r):=r w_1(r)$ is in $C^1([0,+\infty))$ and satisfies $\dot{\xi}(0)=0$.
	\end{enumerate}
\end{lemma}
\begin{proof}
	Let us first set
	\begin{equation*}
	\overline{w}(r)=\max_{\rho\in(0,r]} w(\rho).
	\end{equation*}
	Then $\overline{w}$ is nondecreasing, not smaller than $w$, and tends to $0$ as $r \rightarrow 0$. Next, we define for $r>0$
	\begin{equation*}
	w_0(r)=\frac{1}{r}\int_r^{2r}\overline{w}(\rho)\,d\rho, \ \ \ \ \ w_1(r)=\frac{1}{r} \int_r^{2r} w_0(\rho)\,d\rho,
	\end{equation*}
	and so we set $w_1(0)=0$. We first observe that, since $\overline{w}$ is nondecreasing, the same holds for $w_0$ and $w_1$. Then we have that $w(r)\leq w(r_0)\leq \overline{w}(2r)$, and so $w_0(r)\rightarrow 0$ as $r \rightarrow 0$. Arguing in the same way with $w_1$ we deduce that properties $(i)$ and $(ii)$ hold. To prove $(iii)$, let us set $\xi(r)=rw_1(r)$. Then $\xi\in C^1((0+\infty))$ with derivate $\dot{\xi}(r)=2w_0(2r)-w_0(r)$. Thus $\dot{\xi}(r)\rightarrow 0$ as $r\rightarrow 0$ and so $\xi$ in $C^1$ in the closed half-line $[0,+\infty)$.
\end{proof}
\begin{proof}[Proof of Proposition \ref{17}]
	The implications $(b)\Longrightarrow (c)$ and $(c)\Longrightarrow (a)$ are obvious; so it is enough to prove that $(a)$ implies $(b)$. Given $p\in D^+u(x)$, let us define, for $r>0$,
	\begin{equation}
	w(r)=\max_{\tiny\begin{array}{c}
		y\in \overline{\Omega}\\
		y: |y-x|\leq r
		\end{array}}
	\Big[\frac{u(y)-u(x)-\langle p,y-x\rangle}{|y-x|}\Big]_+,
	\end{equation}
	where $[\cdot]_+$ denotes the positive part. The function $w$ is continuous and tends to $0$ as $r \rightarrow 0$, by the definition of $D^+u$. Let $w_1$ be the function given by the previous lemma. Then, setting 
	$$
	\phi(y)=u(x)+\langle p,y-x\rangle +|y-x|w_1(|y-x|),
	$$
	we have that $\phi\in C^1(\R^n)$ and touches $u$ from above at $x$.
\end{proof}

The idea of the proof is based on \cite[Theorem 4.5]{7}.
Let $x\in\partial \Omega$ and let $\nu(x)$ be the outward unit normal  to $\partial\Omega$ in $x$. Let $\theta\in \R^n$ be such that $\langle \theta,\nu(x)\rangle\leq0$.
Let us set 
\begin{equation*}
M(\theta,x)=\min_{ p\in D^+u(x)}\langle p,\theta\rangle.
\end{equation*}
It suffices to prove that
\begin{align}\label{ax}
\limsup_{\tiny\begin{array}{c}
	h\rightarrow 0^+\\
	\theta'\rightarrow \theta\\
	x+h\theta'\in \overline{\Omega}
	\end{array}} \frac{u(x+h\theta')-u(x)}{h}\leq M(\theta,x)\leq \liminf_{\tiny\begin{array}{c}
	h\rightarrow 0^+\\
	\theta'\rightarrow \theta\\
	x+h\theta'\in \overline{\Omega}
	\end{array}}  \frac{u(x+h\theta')-u(x)}{h}.
\end{align}
The first inequality in \eqref{ax} is straightforward. Indeed, for any $p\in D^+u(x)$,
\begin{equation*}
\limsup_{\tiny\begin{array}{c}
	h\rightarrow 0^+\\
	\theta'\rightarrow \theta\\
	x+h\theta'\in \overline{\Omega}
	\end{array}} \frac{u(x+h\theta')-u(x) - \langle p,h\theta'\rangle}{h}\leq 0.
\end{equation*}
So,
\begin{equation*}
\limsup_{\tiny\begin{array}{c}
	h\rightarrow 0^+\\
	\theta'\rightarrow \theta\\
	x+h\theta'\in \overline{\Omega}
	\end{array}}\frac{u(x+h\theta)-u(x)}{h}\leq \langle p,\theta\rangle, \ \ \ \forall \ p\in D^+u(x).
\end{equation*}
In order to prove the last inequality in \eqref{ax}, pick  sequences $h_k\rightarrow 0$ and $\theta_k\rightarrow \theta$ such that $x+h_k\theta_k\in \overline{\Omega}$ and
\begin{equation}\label{ax1}
\lim_{k\rightarrow \infty}\frac{u(x+h_k\theta_k)-u(x)}{h_k}=\liminf_{\tiny\begin{array}{c}
	h\rightarrow 0^+\\
	\theta'\rightarrow \theta\\
	x+h\theta'\in \overline{\Omega}
	\end{array}}\frac{u(x+h\theta')-u(x)}{h}.
\end{equation}
Let us define
\begin{equation*}
\mathcal{Q}(x,\theta_k)=\Big\{x'\in \Omega : \langle x'-x,\theta_k\rangle>0, |\langle x'-x,\theta_k\rangle\theta_k-(x'-x)|\leq |x'-x|^2\Big\}.
\end{equation*}
We observe that the interior of $\mathcal{Q}(x,\theta_k)$ is nonempty.
Since $u$ is Lipschitz there exists a sequence $x_k$ such that
\begin{enumerate}
	\item[(i)] $x_k\in \mathcal{Q}(x,\theta_k)$, $x_k\rightarrow x$ as $k\rightarrow \infty$;
	\item[(ii)] $u$ is differentiable at $x_k$ and there exists $\overline{p}\in D^+u(x)$ such that $D u(x_k)\rightarrow \overline{p}$ as $k\rightarrow \infty$;
	\item[(iii)] $|s_k-h_k|\leq h_k^2$, where $s_k=\langle x_k-x,\theta_k\rangle$.  
\end{enumerate}
By the Lipschitz continuity of $u$, we note that $(iii)$ yields
\begin{align*}
&\Big| \frac{u(x+h_k\theta_k)-u(x)}{h_k}-\frac{u(x+s_k\theta_k)-u(x)}{s_k}\Big|\leq \frac{|u(x+h_k\theta_k)-u(x+s_k\theta_k)|}{h_k}\\
&+\left|\frac{1}{h_k}-\frac{1}{s_k}\right|\big[|u(x+s_k\theta_k)-u(x)|\big]\leq 2\L(u) h_k.
\end{align*}
So, by \eqref{ax1} we have that
\begin{equation}\label{ax4}
\lim_{k\rightarrow \infty} \frac{u(x+s_k\theta_k)-u(x)}{s_k}=\liminf_{\tiny\begin{array}{c}
	h\rightarrow 0^+\\
	\theta'\rightarrow \theta\\
	x+h\theta'\in \overline{\Omega}
	\end{array}}\frac{u(x+h\theta')-u(x)}{h}.
\end{equation}
Moreover,
\begin{align*}
&u(x+s_k\theta_k)-u(x)=[u(x+s_k\theta_k)-u(x_k)]+ [u(x_k)-u(x)-\langle D u(x_k),x_k-x\rangle]\\
&+\langle D u(x_k),x_k-x-s_k\theta_k\rangle+ \langle s_k D u(x_k),\theta_k\rangle.
\end{align*}
Since $u$ is locally Lipschitz and $x_k\in \mathcal{Q}(x,\theta_k)$, one has that
\begin{align*}
&\big|u(x+s_k\theta_k)-u(x_k)\big|+ \big|\langle D u(x_k),x_k-x-s_k\theta_k\rangle\big|\leq 2\L(u)\big|x_k-x-s_k\theta_k\big|\\
&	\leq 2\L(u) |x_k-x|^2.
\end{align*}
Since $u$ is semiconcave we deduce that
\begin{equation*}
u(x_k)-u(x)-\langle D u(x_k),x_k-x\rangle\geq -C|x_k-x|\omega(|x_k-x|),
\end{equation*}
for some constant $C>0$.
Therefore
\begin{equation*}
\frac{u(x+s_k\theta_k)-u(x)}{s_k}\geq \langle D u(x_k),\theta_k \rangle - \frac{2\L(u)|x_k-x|^2+C|x_k-x|\omega(|x_k-x|)}{s_k}.
\end{equation*}
By the definition of $\mathcal{Q}(x,\theta_k)$ one has that $s_k|\theta_k|\geq |x_k-x|-|x_k-x|^2$, so that, as $x_k\rightarrow x$, $|x_k-x|\leq 2s_k$ for $k$ large enough. Recalling (ii), \eqref{ax4}, and the fact that $\theta_k\rightarrow \theta$, we conclude that
\begin{equation}
\liminf_{\tiny\begin{array}{c}
	h\rightarrow 0^+\\
	\theta'\rightarrow \theta\\
	x+h\theta'\in \overline{\Omega}
	\end{array}}\frac{u(x+h\theta')-u(x)}{h}\geq \langle\overline{p}, \theta\rangle\geq M(\theta,x).
\end{equation}
This completes the proof.

\end{document}